\documentclass[12pt]{article}
\usepackage{fancyhdr}
\usepackage{setspace}
\usepackage{enumitem}
\setlist{nolistsep}
\usepackage[mathscr]{eucal}
\newtheorem{theorem}{Theorem}[section]
\newtheorem{lemma}[theorem]{Lemma}
\newtheorem{corollary}[theorem]{Corollary}

\newtheorem{definition}[theorem]{Definition}

\newenvironment{proof}{\noindent{\bf Proof}\hspace{0.5em}}
    { \null \hfill $\square$ \par}

\usepackage{stmaryrd} \newcommand{\rbose}{\rrbracket}\newcommand{\lbose}{\llbracket}\newcommand{\rround}{\rrparenthesis}\newcommand{\lround }{\llparenthesis}

\usepackage{pifont}
\usepackage{amssymb}
\usepackage{amsmath}
\usepackage{graphicx}

\usepackage[usenames,dvipsnames]{color}

\newcommand{\scroll}{\mathscr S}
\newcommand{\VO}{\V(\Bo{\O})}
\newcommand{\VC}{\V(\Bo{\C})}
\newcommand{\VCplus}{\V(\Bo{\Cplus})}

\newcommand{\VB}{\V(\Bo{\pi})}

    \newcommand{\Fq}{\mathbb F_{q}}
\newcommand{\Fqqq}{\mathbb F_{q^3}}
\newcommand{\Fqq}{\mathbb F_{q^2}}
\newcommand{\Fqqqqqq}{\mathbb F_{q^6}}
\newcommand{\Fqh}{\mathbb F_{q^h}}

\newcommand{\plus}{{\raisebox{.2\height}{\scalebox{.5}{\textup{+}}}}}
\newcommand{\Cplus}{\C^{{\rm \pmb\plus}}}
\newcommand{\si}{\Sigma_\infty}
\renewcommand{\O}{\mathcal O}

   \newcommand{\ocpi}{\bar{\mathsf c}_\pi}

  \newcommand{\ocpis}{\bar{\mathsf c}^2_\pi}

   \newcommand{\ocb}{\bar{\mathsf c}_b}

\newcommand\T{{\cal T}}

\newcommand\U{{\cal U}}

\newcommand{\R}{\mathcal R}
\newcommand\C{{\cal C}}

\newcommand\V{{\mathscr V}}
\newcommand\I{{\cal I}}

\newcommand{\IB}{\I_{\mbox{\tiny\textup{Bose}}}}

\newcommand\A{{\cal A}}
\newcommand\F{{\cal F}}

\renewcommand{\S}{\mathcal S}

\newcommand{\K}{\mathcal K}

\newcommand{\Q}{\mathcal Q}

\renewcommand\setminus{\backslash}
\newcommand{\st}{\,|\,}

\newcommand\PGammaL{{\mbox{P$\Gamma$L}}}
\newcommand\PGL{{\rm PGL}}

\newcommand\PG{{\rm PG}}

\newcommand\Bo[1]{\mbox{$\lbose#1\rbose$}}

\newcommand{\sistar}{{\Sigma}^{\mbox{\tiny\ding{73}}}_{\mbox{\raisebox{.4\height}{\scalebox{.6}{$\infty$}}}}}

\newcommand{\Pirstar}{{\Pi}^{\mbox{\tiny\ding{73}}}_r}
\newcommand{\Pirstarstar}{{\Pi}^{\mbox{\tiny\ding{72}}}_r}

\newcommand{\PiVstar}{{\Pi}^{\mbox{\tiny\ding{73}}}_{\V}}
\newcommand{\PiPstar}{{\Pi}^{\mbox{\tiny\ding{73}}}_{\mbox{\raisebox{.4\height}{\scalebox{.5}{$P$}}}}}

\newcommand{\Sigmarstarstar}{{\Sigma}^{\mbox{\tiny\ding{72}}}_r}

\newcommand{\Qonestar}{{\Q}^{\mbox{\tiny\ding{73}}}_1}
\newcommand{\Qninestar}{{\Q}^{\mbox{\tiny\ding{73}}}_9}
\newcommand{\Qzerostar}{{\Q}^{\mbox{\tiny\ding{73}}}_0}
\newcommand{\Qtwostar}{{\Q}^{\mbox{\tiny\ding{73}}}_2}
\newcommand{\Qonestarstar}{{\Q}^{\mbox{\tiny\ding{72}}}_1}
\newcommand{\Qzerostarstar}{{\Q}^{\mbox{\tiny\ding{72}}}_0}
\newcommand{\Qtwostarstar}{{\Q}^{\mbox{\tiny\ding{72}}}_2}

\newcommand{\Vonestar}{{\V}^{\mbox{\tiny\ding{73}}}_1}
\newcommand{\Vtwostar}{{\V}^{\mbox{\tiny\ding{73}}}_2}
\newcommand{\Vthreestar}{{\V}^{\mbox{\tiny\ding{73}}}_3}

\renewcommand{\star}{{^{\mbox{\tiny\ding{73}}}}}
\newcommand{\blackstar}{{^{\mbox{\tiny\ding{72}}}}}

\newcommand{\Kstar}{{\K^{\mbox{\tiny\ding{73}}}}}
\newcommand{\Kstarstar}{{\K^{\mbox{\tiny\ding{72}}}}}

\newcommand{\Vstar}{{\V^{\mbox{\tiny\ding{73}}}}}

\newcommand{\Qstar}{{\Q^{\mbox{\tiny\ding{73}}}}}

\newcommand{\SC}{\mathscr S}

\newcommand\conjB[1]{{  #1^{\mathsf c_{\scalebox{0.5}{\mbox{$\pi$}}}}}}
\newcommand\conjBsq[1]{{  #1^{\mathsf c^2_{\pi}}}}

\newcommand{\Label}{\label}

\newcommand{\Pit}{\Gamma}

\begin{document}
%

\title{The Bose representation of $\PG(2,q^3)$ in $\PG(8,q)$}
\author{S.G. Barwick, Wen-Ai Jackson and Peter Wild}
\date{}
\maketitle
%
%
%
%
 
%


\begin{abstract} This article looks at the Bose representation of $\PG(2,q^3)$ as a 2-spread of $\PG(8,q)$. It is shown that an $\Fq$-subline of $\PG(2,q^3)$ corresponds to a 2-regulus, and an $\Fq$-subplane corresponds to a Segre variety $\S_{2;2}$. Moreover, the extension of these varieties to $\PG(8,q^3)$ and $\PG(8,q^6)$ is determined. These are used to determine the structure of an $\Fq$-conic of $\PG(2,q^3)$ in the Bose representation in  $\PG(8,q)$. 
\end{abstract}

\section{Introduction}

  In \cite{Bose}, Bose gave a representation of $\PG(2,q^2)$ as a regular 1-spread in $\PG(5,q)$. It is straightforward to generalise this to represent $\PG(2,q^3)$ as a regular 2-spread in $\PG(8,q)$, and we examine this representation in detail. 
  In particular, we determine the representation of conics, $\Fq$-sublines, $\Fq$-subplanes and $\Fq$-conics of $\PG(2,q^3)$ in $\PG(8,q)$.  
  Moreover, we determine the extensions of the corresponding varieties to $\PG(8,q^3)$ and $\PG(8,q^6)$.   
   Our motivation in looking at   these extensions is to study the representation of $\Fq$-conics of $\PG(2,q^3)$ in the Bruck-Bose $\PG(6,q)$ representation. In particular, \cite{BJW-4} uses these extensions to characterise $\Fq$-conics of $\PG(2,q^3)$ as corresponding to certain  normal rational curves of $\PG(6,q)$.

  The article is set out as follows. In Section~\ref{sec:2}, we introduce the background and notation we use. In Section~\ref{sec:prelim}, we prove some preliminary results. First  we 
  need several results relating to disjoint planes in $\PG(8,q)$. Secondly, in Section~\ref{sec:scroll}, we carefully define a notion of 
 a scroll  in $\PG(8,q)$ ruled by three planar varieties.
  Section~\ref{sec3.4} is devoted to the variety associated with  a scroll ruled by three conics, and we determine the order and dimension of this variety.
   
 In Section~\ref{sec:coord}, we investigate coordinates for the Bose representation. 
 We calculate coordinates for the transversal planes of the regular 2-spread. We need a suitable description of certain points in $\PG(8,q^3)$, and we determine these by looking at the conjugacy map with respect to an $\Fq$-subplane. 
 
 Section~\ref{sec:var} looks at a variety of $\PG(2,q^3)$, and  uses coordinates to describe  the corresponding variety in $\PG(8,q)$. 
A geometric description of this  variety is given. 
 An application to  the representation of a non-degenerate conic of $\PG(2,q^3)$ in the Bose representation is given in 
Theorem~\ref{Fqqq-conic-Bose}. 

The machinery that has been developed in the article is then used to look at $\Fq$-structures of $\PG(2,q^3)$ in $\PG(8,q)$; and to  determine the extension of the resulting varieties to $\PG(8,q^3)$ and $\PG(8,q^6)$.  
Section~\ref{sec:lines-plane} determines the representation of $\Fq$-sublines and $\Fq$-subplanes of $\PG(2,q^3)$ in the Bose representation, and  Section~\ref{sec:conic} determines the representation of $\Fq$-conics of $\PG(2,q^3)$ in the Bose representation.

\section{Background and Notation}\Label{sec:2}

\subsection{Background}

We denote the unique finite field of prime power order $q$ by $\Fq$. An $\Fq$-subplane of $\PG(2,q^3)$ is a subplane that is projectively equivalent to $\PG(2,q)$. Similarly, an $\Fq$-subline is a subline  that is projectively equivalent to $\PG(1,q)$.

The Frobenius map $x\mapsto x^q$ for $x\in \mathbb F_{q^{h}}$ gives rise to an 
 automorphic collineation in $\PGammaL(n,q^h)$ of order $h$ acting on points of $\PG(n,q^{h})$
that fixes the points of $\PG(n,q)$ pointwise, that is
$
X=(x_0,\ldots,x_{n})\ \longmapsto  X^{q}=(x_0^q,\ldots,x_{n}^q).
$
We say the points $X,X^q\ldots,X^{q^h-1}$ are conjugate points with respect to the conjugacy map $X\mapsto X^q$.

We work with Segre varieties, see \cite[Section 4.5]{HT-new} for details. In particular, the Segre variety $\S_{2;2}$ in $\PG(8,q)$  contains two systems of  maximal subspaces (planes), denoted $\R$ and $\R'$, such that every plane in $\R$ meets every plane in $\R'$ in a point. 
A 2-{\em regulus} of $\PG(5,q)$ is the system of maximal 2-spaces
of a Segre variety ${\mathscr
  S}_{1;2}$, see \cite[Section 4.6]{HT-new}. 
  
  A 2-{\em spread} of $\PG(8,q)$ is a set of  planes that partition the points of 
$\PG(8,q)$.  
We use the following construction of a regular 2-spread of $\PG(3s+2,q)$, see \cite{CasseOKeefe}.
Embed $\PG(8,q)$ in $\PG(8,q^3)$ and 
consider the collineation  $X=(x_0,\ldots,x_{8}) \longmapsto 
X^{q}=(x_0^q,\ldots,x_{8}^q)$ acting on  $\PG(8,q^3)$. 
 Let $\Gamma $ be a plane in $\PG(8,q^3)$ which is disjoint from $\PG(8,q)$,  such that $\Gamma,\Gamma^{q},\Gamma^{q^{2}}$ span $\PG(8,q^3)$ (so any two span a 5-space which is disjoint from the third). 
For a point $X\in\Gamma$, the plane $\langle X,X^{q}, X^{q^2}\rangle$ of $\PG(8,q^3)$ meets $\PG(8,q)$ in a plane. The planes $\langle X,X^{q}, X^{q^2}\rangle\cap\PG(8,q)$ for $X\in\Gamma$ form a regular 2-spread of $\PG(8,q)$. The planes  $\Gamma$, $\Gamma^{q}$ and $\Gamma^{q^2}$ are called 
 the three {\em transversal spaces} of the 2-spread. Conversely, any regular  2-spread of $\PG(8,q)$ has a unique set of three  transversal $s$-spaces in $\PG(8,q^3)$, and can be constructed in this way, see \cite[Theorem 6.1]{CasseOKeefe}.  

\subsection{Variety-extensions}
A variety in $\PG(n,q)$ has a natural extension to a variety in the cubic extension $\PG(n,q^3)$ and to a variety in $\PG(n,q^6)$ (a sextic extension of $\PG(n,q)$, and a quadratic extension of $\PG(n,q^3)$).  If $\K$ is a variety of $\PG(n,q)$, then the pointset of $\K$ is the set of points of $\PG(n,q)$ satisfying  the set $\F=\{f_i(x_0,\ldots,x_n)=0, \ i=1,\ldots,k\}$ of $k$ homogeneous $\Fq$-equations $f_i$.
We define the  {\em variety-extension}  $\Kstar$ of $\K$ to $\PG(n,q^3)$, to be the set of points of $\PG(n,q^3)$ that satisfy the same set $\F$  of homogeneous equations as $\K$.  
  Similarly, we can define the {\em variety-extension}  $\Kstarstar$ of $\K$ to $\PG(n,q^6)$.
So if $\Pi_r$ is an $r$-dimensional subspace of $\PG(n,q)$, then $\Pirstar$ is the natural extension to an $r$-dimensional subspace of $\PG(n,q^3)$, and $\Pirstarstar$ is the extension to $\PG(n,q^6)$. Moreover, if $\Sigma_r$ is an $r$-dimensional subspace of $\PG(n,q^3)$ (possibly disjoint from $\PG(n,q)$), then $\Sigmarstarstar$ denotes the extension to $\PG(n,q^6)$. 
In this article we use the $\star$ and $\blackstar$ notation for varieties in the  Bose representations, that is, when $n=8$. We do not use the $\star$ and $\blackstar$ notation in $\PG(2,q^3)$.

\subsection{The Bose representation of $\PG(2,q^3)$ in $\PG(8,q)$}

Bose \cite{Bose} gave a construction to represent $\PG(2,q^2)$ using a regular 1-spread in $\PG(5,q)$. This construction generalises to the Bose representation of $\PG(2,q^h)$ using a regular  $(h-1)$-spread in $\PG(3h-1)$. We consider $h=3$, that is, the Bose representation of $\PG(2,q^3)$ using a regular  2-spread $\mathbb S$  in $\PG(8,q)$. 
 Let $\IB$ be the incidence structure with {\em points} the $q^6+q^3+1$  planes of  ${\mathbb S}$; {\em lines} the  5-spaces of $\PG(8,q)$ that meet ${\mathbb S}$ in $q^3+1$ planes; and {\em incidence} is inclusion. 
The  
5-spaces of $\PG(8,q)$ that meet ${\mathbb S}$ in $q^3+1$ planes form a dual spread $\mathbb H$  (that is, each 7-space of $\PG(8,q)$ contains a unique 5-space in $\mathbb H$). 
  Then $\IB\cong\PG(2,q^3)$, and this representation is called the {\em Bose representation} of $\PG(2,q^3)$ in $\PG(8,q)$. 
 The regular 2-spread $\mathbb S$ has three conjugate transversal planes which we denote throughout this article by $\Pit$, $\Pit^q$, $\Pit^{q^2}$. Note that $\IB\cong\Pit \cong\PG(2,q^3)$.

We use the following notation. 
 Let $\bar X$ be a point of $\PG(2,q^3)$, then the Bose representation of $\bar X$ is  a plane of $\mathbb S$ denoted by $\Bo{X}$. 
 In $\PG(8,q^3)$, we have $\Bo{X}\star\cap\Pit=X$ and  $\Bo{X}\star=\langle  X, X^q,X^{q^2}\rangle$. Thus 
 $\bar X$ corresponds to a unique point $X$ of $\Pit$, and the points of $\Pit$
 and $\PG(2,q^3)$ are in one-to-one correspondence. 
 
  More generally, if $\bar\K$ is a set of points of $\PG(2,q^3)$, then $\lbose \K\rbose=\{\Bo{X}\st \bar X\in\bar \K\}$ denotes the corresponding set of planes in the Bose representation in $\PG(8,q)$, and $\K=\{\Bo{X}\star\cap\Pit\st \bar X\in\bar \K\}$ denotes the corresponding set of points of $\Pit$ .

 So we have the following correspondences:
 $$
\begin{array}{ccccccccccc}
\PG(2,q^3)&\cong&\Pit&\cong&\IB\\
\bar X &\longleftrightarrow&
X
&\longleftrightarrow&
\Bo{X}=\langle  X, X^q,X^{q^2}\rangle\cap\PG(8,q).
\end{array}$$

\subsection{Notation summary}

 \begin{itemize}
  \item [$\bullet$] For a variety $\K$ in $\PG(n,q)$, $n>2$, denote the variety-extension to $\PG(n,q^3)$ by $\Kstar$, and the extension to $\PG(n,q^6)$ by $\K\blackstar$. 
  
  \item [$\bullet$] For a point $X=(x_0,\ldots,x_{n})$ of $\PG(n,q^t)$, let 
$X^{q}=(x_0^q,\ldots,x_{n}^q).$

  \item [$\bullet$] $\mathbb S$ is a regular 2-spread in $\PG(8,q)$.
  
   \item [$\bullet$] $\mathbb S$ has three transversal planes in $\PG(8,q^3)$, denoted $\Pit$, $\Pit^q$, $\Pit^{q^2}$.

  \item [$\bullet$] A point $\bar X$ in $\PG(2,q^3)$ corresponds to a point $X$ in the transversal plane $\Pit$, and to a plane $\Bo{X}=\langle  X, X^q,X^{q^2}\rangle\cap\PG(8,q)$ of $\mathbb S$.

\end{itemize}

 \section{Preliminary Results}\Label{sec:prelim}
  
  \subsection{Planes of $\PG(8,q)$}
  
  \begin{lemma}\Label{lem:planes1} 
  Let $\alpha,\beta,\gamma$ be three planes which span $\PG(8,q)$. Let $P$ be a point of $\PG(8,q)$ not on a line meeting two of $\alpha,\beta,\gamma$.
  Then $P$ lies on a unique plane that meets each of $\alpha,\beta,\gamma$.
  \end{lemma}

\begin{proof} 
Let $P$ be a point of  $\PG(8,q)$ which is not on a line joining two of $\alpha,\beta,\gamma$.
So $\langle P,\alpha\rangle$ is a 3-space  that does not meet $\beta$ or $\gamma$. Hence $\Sigma_6=\langle P,\alpha,\beta\rangle$ is a 6-space. As $\alpha,\beta,\gamma$ span $\PG(8,q)$, $\Sigma_6$ meets  $\gamma$ in a point we denote by $Q$. So $\Sigma_4=\langle P,Q,\alpha\rangle$ is a 4-space contained in $\Sigma_6$. As $\alpha,\beta$ span a 5-space contained in $\Sigma_6$,  $\Sigma_4$ meets $\beta$ in a point denoted $R$. The two planes $\pi=\langle P,Q,R\rangle$ and $\alpha$ lie in  the 4-space $\Sigma_4$, and so meet in a point $S$. That is, $\pi$ meets $\alpha$ in the point $S$, meets $\beta$ in the point $R$ and meets $\gamma$ in the point $Q$. That is, $P$ lies on at least one plane that meets each of $\alpha,\beta,\gamma$.

Suppose $P$  lies in two distinct planes $\pi_1,\pi_2$  that meet each of $\alpha,\beta,\gamma$, so $\langle \pi_1,\pi_1\rangle$ has dimension 3 or 4. Consider the set of six (possibly repeated)  points $\K=\{\pi_i\cap\alpha,\,\pi_i\cap\beta,\,\pi_i\cap\gamma\st i=1,2\}$. 
Suppose $\langle \pi_1,\pi_2\rangle$ is a 3-space, so $\pi_1\cap\pi_2=\ell$ is a line. As $P\in\ell$, by assumption at most one of the planes $\alpha,\beta,\gamma$ meets $\ell$. Hence $|\K|\geq 5$, so the 3-space $\langle \pi_1,\pi_2\rangle$ meets two of the planes $\alpha,\beta,\gamma$  in a line, and meets the other in at least a point. This contradicts the three  planes $\alpha,\beta,\gamma$  spanning $\PG(8,q)$. 
If  $\langle \pi_1,\pi_2\rangle$ is a 4-space, then $|\K|=6$, and so $\langle \pi_1,\pi_2\rangle$ meets each of $\alpha,\beta,\gamma$ in a line, contradicting the  three  planes $\alpha,\beta,\gamma$  spanning $\PG(8,q)$. Hence $P$ lies 
 on at most one plane  that meets each of $\alpha,\beta,\gamma$. We conclude that $P$ lies on exactly one plane that meets each of $\alpha,\beta,\gamma$.
 \end{proof}

 \begin{lemma}\Label{lem:HT2} Let $\alpha,\beta,\gamma,\delta$ be four planes of $\PG(8,q)$, such that any three span $\PG(8,q)$. 
 Then $\alpha,\beta,\gamma,\delta$ are contained in a unique Segre variety $\S_{2;2}$. 
  \end{lemma}
 
 \begin{proof}
By Lemma~\ref{lem:planes1},
through each point $P_i\in\alpha$, there is a plane $\pi_i$ that meets each of $\beta,\gamma,\delta$ in a point. 
As  $\beta,\gamma,\delta$  span $\PG(8,q)$, the planes $\pi_i$, $i=1,\ldots,q^2+q+1$  are pairwise disjoint. 
Hence there is a unique set $\R=\{\pi_1,\ldots,\pi_{q^2+q+1}\}$ of $q^2+q+1$ pairwise disjoint planes that each meet $\alpha,\beta,\gamma,\delta$ in a point. 
We show that  $\R$ is one set of maximal spaces of a Segre variety $\S_{2;2}$. 

 Let $U_i$, $i=0,\ldots,8$,  be the point of $\PG(8,q)$ with a 1 in entry $i$, and zeroes elsewhere. A generalisation of the proof that three pairwise disjoint lines in $\PG(3,q)$ lie in a unique regulus (see for example \cite[Theorem 3.5]{UnitalBook}) shows that 
  without loss of generality, we can coordinatise so that 
   $\alpha=\langle U_0,U_1,U_2\rangle$, $\beta=\langle U_3,U_4,U_5\rangle$, $\gamma=\langle U_6,U_7,U_8\rangle$, and 
  $\delta=\langle D_0,D_1,D_2\rangle$ with 
  $D_0=(1,0,0,1,0,0,1,0,0) $, $D_1=(0,1,0,0,1,0,0,1,0)$, $  
 D_2=(0,0,1,0,0,1,0,0,1).$
 Let $A\in\alpha$,  so $A=(a_0,a_1,a_2,0,0,0,0,0,0)$ for some $a_0,a_1,a_2\in\Fq$, not all zero. 
 Let $\pi_A$ be the unique plane of $\R$ through $A$. Let $B=\pi_A\cap\beta$, $C=\pi_A\cap\gamma$  and  $D=\pi_A\cap\delta$. So $B=(0,0,0,b_0,b_1,b_2,0,0,0)$ for some $b_0,b_1,b_2\in\Fq$, not all zero, and  $C=\pi_A\cap\gamma=(0,0,0,0,0,0,c_0,c_1,c_2)$ for some $c_0,c_1,c_2\in\Fq$, not all zero.
As $D=\pi_A\cap\delta$, we have $$D=r_0A+r_1B+r_2C=s_0D_0+s_1D_1+s_2D_2$$ for some 
$r_0,r_1,r_2,s_0,s_1,s_2\in\Fq$. 
That is, $$(r_0a_0,r_0a_1,r_0a_2,r_1b_0,r_1b_1,r_1b_2,r_2c_0,r_2c_1,r_2c_2)=(s_0,s_1,s_2,s_0,s_1,s_2,s_0,s_1,s_2)$$
and so $B=(0,0,0,a_0,a_1,a_2,0,0,0)$ and $D=(r_0a_0,r_0a_1,r_0a_2,\ r_1a_0,r_1a_1,r_1a_2,\ r_2a_0,r_2a_1,r_2a_2)$.  
Hence as in \cite[page 209]{HT-new}, $\pi_A=\langle A,B,D\rangle$ is a plane in one maximal system of a Segre variety $\S_{2;2}$.
\end{proof}

\subsection{T-points, T-lines and T-planes}

We need some properties of  planes of $\PG(8,q^3)$ that meet all three transversal planes $\Pit$,    $\Pit^q$, $\Pit^{q^2}$ of the Bose spread $\mathbb S$. We call a plane  of $\PG(8,q^3)$ that meets all three transversal planes a \emph{T-plane};  call a line  of $\PG(8,q^3)$ that meets two transversal planes a \emph{T-line}; and call a point of $\PG(8,q^3)$  that lies in a  transversal plane a \emph{T-point}.

\begin{lemma}\Label{lem:hypcong} 
Two T-planes  of \ $\PG(8,q^3)$   are either equal, disjoint, meet in  a T-point, or meet in a T-line.  
\end{lemma}

\begin{proof} 
Let $P$ be a point of  $\PG(8,q^3)$ which is not a T-point, and does not lie on a T-line. 
As $P$ is not on a  T-line, by Lemma~\ref{lem:planes1}, 
 $P$ lies on a unique T-plane.
The result follows. \end{proof}

 \subsection{Scrolls}\Label{sec:scroll}
 
The notion of a   scroll that rules two normal rational curves according to a projectivity in $\PGL(2,q)$  is well known.
Here, we are interested in scrolls which rule \emph{three} planar varieties, that is, scrolls that are ruled by planes. 
We give a general definition of scrolls.

Let $r,d$ be positive integers. For $i=0,\dots,r$, let $n_i$ be a positive integer and let $\V_i$ be a variety of $\PG(n_i,q)$ of dimension $d$. Suppose that there exists a primal $\V$ of $\PG(d+1,q)$ such that, for $i=0,\dots,r$, $\V_i$ is in one-to-one correspondence with $\V$ under a birational mapping  $\psi_i\colon\V \rightarrow \V_i$. Let $\phi_i$, $i=1,\dots,r$ be projectivities of $\PG(d+1,q)$ fixing $\V$.
Consider $\PG(n,q)$ with $n=n_0+\cdots+n_r +r-1$ and suppose that $\V_0,\dots,\V_{r}$ are varieties contained in mutually disjoint subspaces of $\PG(n,q)$. For each point $P \in \V_0$ let $\Pi_P$ be the $r$-dimensional subspace $\langle P,P^{\psi_0^{-1}\phi_1\psi_1},\dots,P^{\psi_0^{-1}\phi_{r}\psi_{r}}\rangle$ (that is, the span of the point of each subspace corresponding to $P$). Define  $\scroll=\scroll(\V_0,\dots,\V_{r})$ to be the set of subspaces $\{\Pi_{P} | P \in \V_0 \}$. We say $\scroll$ is a \emph{scroll of type $(r,d)$} or an \emph{$(r,d)$-scroll.}

\emph{Example 1. Normal Rational Curves} 
 Let $r=1,d=1$. Suppose that $\V_0$ and $\V_1$ are normal rational curves of $\PG(n_0,q)$ and $\PG(n_1,q)$ respectively. For $i=0,1$, let $$\V_i=\{(s^{n_i},s^{n_i-1}t,\dots,st^{n_i-1},t^{n_i}) | s,t \in \Fq \}$$ and let $\psi_i$ be the mapping from the line $z=0$ of $\PG(2,q)$ to $\V_i$ given by $$\psi_i\colon(s,t,0) \longmapsto (s^{n_i},s^{n_i-1}t,\dots,st^{n_i-1},t^{n_i})$$ for $s,t\in\Fq$. Embedding $\V_0$ and $\V_1$ in independent subspaces of $\PG(n_0+n_1+1,q)$ and taking $\phi$ to be the identity we have that the collection $\scroll(\V_0,\V_1)$ of lines
$$\{(as^{n_0},as^{n_0-1}t,\dots,ast^{n_0-1},at^{n_0},bs^{n_1},bs^{n_1-1}t,\dots,bst^{n_1-1},bt^{n_1}) | a,b \in \Fq \}$$
 for $s,t \in \Fq$, $(s,t)\not=(0,0)$, is a scroll of type $(1,1)$. 
 It is well known that $\scroll(\V_0,\V_1)$ is a variety of dimension $2$. We can see this using techniques of \cite{SR} as follows. 
Supposing $n_1 \ge n_0$, the pointset of $\scroll(\V_0,\V_1)$ may also be written as the set of points $$(as^{n_1},as^{n_1-1}t,\dots,as^{n_1-n_0+1}t^{n_0-1},as^{n_1-n_0}t^{n_0},bs^{n_1},bs^{n_1-1}t,\dots,bst^{n_1-1},bt^{n_1}) $$ for $ a,b,s,t \in \Fq, s\not=0$, together with  $(0,\dots,0,at^{n_1},0,\dots,0,bt^{n_1})$ for $ t,a,b \in \Fq, t\not=0.$
 This pointset  is in one-to-one algebraic correspondence with the hyperbolic quadric 
 \begin{eqnarray*}
 V(x_0x_3-x_1x_2)&=&\{(s,t,a,b) | s,t,a,b \in \Fq, sb=ta \}\\
 &=&\{(sa,sb,ta,tb) | s,t,a,b \in \Fq \}
 \end{eqnarray*} of $\PG(3,q)$ since they are both in one-to-one correspondence with $$\PG(1,q) \times \PG(1,q) = \{\big((s,t),(a,b)\big) | s,t,a,b \in \Fq \}.$$ Hence the pointset is a variety of dimension 2.

\subsection{Scroll-extensions}\label{sec3.4}

We will be working with $(2,2)$-scrolls in $\PG(8,q)$, and define a natural extension to $\PG(8,q^3)$ and $\PG(8,q^6)$. 

First consider a variety $\V$  in a plane $\Pi_\V$ of $\PG(8,q)$, so the points of $\V$ satisfy a set $\F$ of equations in $\Pi_\V$. Note that $\V$ is also a variety of $\PG(8,q)$, it consists of the points that satisfy both the equations in $\F$  and   the equations defining $\Pi_\V$. We define the extension of $\V$ to $\PG(8,q^3)$, denoted $\V\star$, to be the set of points in the extended plane $\PiVstar$ that satisfy the same set of equations that define $\V$. 
   
   Suppose $\scroll(\V_1,\V_2,\V_3)$ is a $(2,2)$-scroll in $\PG(8,q)$. 
The    birational mappings  $\psi_1,\psi_2,\psi_3$  and the projectivities $\phi_1$, $\phi_2$  associated with $\scroll(\V_1,\V_2,\V_3)$ act on points of the form $(x,y,z)$, $x,y,z\in\Fq$. They have unique extensions to act on points of the form $(x,y,z)$, $x,y,z\in\Fqqq$. 
For $P\in\Vonestar$, let    $\PiPstar$ be the plane $\langle P,P^{\psi_0^{-1}\phi_1\psi_1},P^{\psi_0^{-1}\phi_2\psi_2}\rangle$. 
The \emph{scroll-extension} of $\scroll(\V_1,\V_2,\V_3)$ to $\PG(8,q^3)$ is a $(2,2)$-scroll $\scroll(\Vonestar,\Vtwostar,\Vthreestar)$
consisting of the  set of subspaces $\{\PiPstar | P \in \Vonestar \}$.
   Similarly, we can define the scroll-extension  to $\PG(8,q^6)$ (and similarly can define   the scroll-extension of a scroll from $\PG(5,q)$ to $\PG(5,q^3)$ and $\PG(5,q^6)$).

  \subsection{A scroll ruling three conics}
  
  We now  show that the points of a scroll of $\PG(8, q)$ that rules three conics forms a variety,  and determine the order and dimension using techniques given in Semple and Roth \cite[I.4]{SR}.

\begin{lemma}\Label{conic-v63} 
 In $\PG(8,q)$, let $\Pi$, $\Pi'$, $\Pi''$  be three disjoint planes which together span $\PG(8,q)$.  Let $\C$, $\C'$, $\C''$ be non-degenerate conics in $\Pi,\Pi',\Pi''$ respectively. Then the pointset of any  scroll $\scroll(\C,\C',\C'')$ form 
 a variety of dimension 3 and order 6.
\end{lemma}

\begin{proof}
Denote the set of points on the scroll $\scroll(\C,\C',\C'')$ by $\SC$. 
 Without loss of generality, we can coordinatise so that  
 \begin{eqnarray*}
 \C&=&\{P_{r,s}=(r^2,rs,s^2,0,0,0,0,0,0)\st r,s\in\Fq\},\\
 \C'&=&\{P'_{r,s}=(0,0,0,r^2,rs,s^2,0,0,0)\st r,s\in\Fq\}\\ \C''&=&\{P''_{r,s}=(0,0,0,0,0,0,r^2,rs,s^2)\st r,s\in\Fq\}.
 \end{eqnarray*}
 and so that the homographies $\phi_i$ are essentially the identity. That is, without loss of generality we coordinatise so that $\SC$ contains the $q+1$ planes $\langle P_{r,s},\ P'_{r,s},\ P''_{r,s}\rangle$ for $r,s\in\Fq$ (not both 0). 
Thus the points of $\SC$ form a variety denoted $\V(\SC)$. We determine the order and dimension of $\V(\SC)$ using techniques given in Semple and Roth \cite[I.4]{SR}. 
We  show that $\V(\SC)$ has dimension 3 by showing that the points of $\SC$ are in one-to-one algebraic correspondence with the points of a 3-space in 
$\PG(4,q)$.

Using the notation of \cite{SR}, let $M(y_0,y_1,y_2,y_3,y_4)=y_4$, the points of $\PG(4,q)$ satisfying $M=0$ form a  3-space which we denote $M_3$. Note that  $M_3$ is an irreducible primal of dimension 3 and order 1. 
Consider the following nine homogeneous cubic equations $F_i=F_i(y_0,y_1,y_2,y_3,y_4)=0$  where $F_0=y_0^3$, $ F_1=y_0^2y_1$, $  F_2=y_0y_1^2$, $  F_3=y_0^2y_2$, $  F_4=y_0y_1y_2,$ $ F_5=y_1^2y_2$, $  F_6=y_0^2y_3$, $  F_7=y_0y_1y_3$, $  F_8=y_1^2y_3.$
Consider the map $\sigma\colon \PG(4,q)\rightarrow \PG(8,q)$ such that 
\begin{eqnarray*}
\sigma(y_0,y_1,y_2,y_3,0)&=&\big(F_0(y_0,y_1,y_2,y_3,0),\ldots, F_8(y_0,y_1,y_2,y_3,0)\big)\\
 &=&(y_0^3,\ y_0^2y_1,\ y_0y_1^2,\ \  y_0^2y_2,\ y_0y_1y_2,\ y_1^2y_2,\ \  y_0^2y_3,\ y_0y_1y_3,\ y_1^2y_3)\\
&=& y_0(y_0^2,\ y_0y_1,\ y_1^2,\ 0,0,0,\ 0,0,0)\ +\ y_2( 0,0,0,y_0^2,\ y_0y_1,\ y_1^2,\ 0,0,0)\ \\&&\quad +\ y_3(0,0,0,\ 0,0,0,\ y_0^2,\ y_0y_1,\ y_1^2).
\end{eqnarray*} 
The map $\sigma$
is an algebraic one-to-one correspondence between  the points of the 3-space $M_3$ and the points of $\SC$. Hence the points of $\S$ form a variety   $\V(\SC)$ which has dimension 3. 
 Note that the kernel of $\sigma$ is $\ker\sigma=\{(0,1,0,0,0), (0,0,a,b,0)\,|\, a,b\in\Fq\}$. 

 Next we show that $\V(\SC)$ has order $6$, that is, we show that  a generic 5-space of $\PG(8,q)$ meets the pointset $\SC$ in six points.
A 5-space in $\PG(8,q)$ is determined by the intersection of three hyperplanes.  
Consider three hyperplane $\Pi_i$, $i=1,2,3$ of $\PG(8,q)$ of equation 
 $ a_{i0}x_0+\ldots+a_{i8}x_8=0$ respectively. The point $$(y_0^3,\ y_0^2y_1,\ y_0y_1^2,\ y_0^2y_2,\ y_0y_1y_2,\ y_1^2y_2,\ y_1^2y_2,\ y_0^2y_3,\ y_0y_1y_3,\ y_1^2y_3)$$ of $\SC$ lies on the hyperplane $\Pi_i$ if
 $a_{i0}y_0^3+a_{i1}y_0^2y_1+a_{i2}y_0y_1^2+ a_{i3}y_2y_0^2+a_{i4}y_2y_0y_1+a_{i5}y_2y_1^2+a_{i6}   y_0^2y_3+a_{i7}y_0y_1y_3+a_{i8} y_1^2y_3=0,
$
and we denote this equation  by $g_i(y_0,y_1,y_2,y_3)=0$.
 The set of points of $\SC$ satisfying $g_i=0$ for some $i\in\{1,2,3\}$ corresponds to a set of points $\K_i$  in the 3-space $M_3$. So   $\K_i$ is a cubic primal in $\PG(3,q)$, that  is a variety $\V^3_2$.
Note that each primal  $\K_i$, $i=1,2,3$  contains the kernel of $\sigma$. Further, for each $i\in\{1,2,3\}$ the four partial derivatives of the equation $g_i$ vanish for the points $(0,0,a,b)$, $a,b\in\Fq$. That is, the line $\ell=\{(0,0,a,b)\st a,b\in\Fq\}$ is a singular line of each primal $\K_1,\K_2,\K_3$. 
As each plane meets a cubic primal $\V^3_2$ in a cubic curve $\V^3_1$, each plane through $\ell$ meets $\K_i$ in either the line $\ell$ counted thrice, or $\ell$ counted twice and another line. 

To determine the number of points of $\SC$ in a generic 5-space of $\PG(8,q)$, we determine the number of points in  
$\K_1\cap\K_2\cap\K_3$ which are not in the kernel of $\sigma$, with $\K_1,\K_2,\K_3$ in general position in relation to each other. 
To abbreviate the notation, we define  nine functions $F_i=F_i(\theta)$, $G_i=G_i(\theta )$, $H_i=H_i(\theta )$, for $i=1,2,3$, where 
$$F_i=a_{i0}\theta ^2+a_{i1} \theta  +a_{i2},\quad G_i=a_{i3}\theta ^2+a_{i4} \theta  +a_{i5},\quad H_i=a_{i6}\theta ^2+a_{i7} \theta  +a_{i8} $$
As we are looking at the generic case,  we can assume $y_1\neq 0$ (since the plane $y_1=0$ containing a point of $\K_1\cap\K_2\cap\K_3$ imposes a condition on the $a_{ij}$, contradicting the surfaces being in general position). We can parameterise the points of $\K_1$ with $y_1\neq0$ 
by letting $y_1=1$, $y_0=\theta $, $y_2=\phi$, so $g_1(\theta ,1,\phi,w)=\theta  F_1+\phi G_1+y_3 H_1$, and    $-y_3H_1=\theta  F_1+\phi G_1$. 
If $H_1(\theta )=0$, this either gives the point $(0,1,0,1)$ which lies in the kernel of $\sigma$, and so is not of interest to us; or it forces a linear relationship between the coefficients $a_{ij}$, contradicting the surfaces being in general position.
That is,  we can assume  
 $H_1(\theta )\neq0$, 
 and so the points on $\K_1$ with $y_1\neq 0$ are  
 $$P_{\theta ,\phi}=(\theta  H_1,\ H_1,\ \phi H_1,-\theta  F_1-\phi G_1).$$
To find the intersection of $\K_1$ and $\K_2$, we substitute the point $P_{\theta ,\phi}$ into the equation $g_2=0$, giving the following curve of order nine, $$H_1^2\big(\theta  H_1F_2+\phi H_1G_2-\theta  H_2 F_1-\phi H_2G_1\big)=0.$$

As $H_1\neq0$ for points on $\K_1$, we have $\phi (H_2G_1-H_1G_2)=\theta (H_1F_2-H_2F_1) .$ So provided $H_2G_1-H_1G_2\neq0$, 
$$\phi= \theta  \frac{H_1F_2-H_2F_1}{H_2G_1-H_1G_2} $$
So points on $\K_1\cap\K_2$ have form: 
$$Q_\theta =P_{\theta ,\phi}=\big(\theta (H_2G_1-H_1G_2),\ (H_2G_1-H_1G_2),\  \theta (H_1F_2-H_2F_1),\ \theta  (F_1G_2-F_2G_1) \big) $$
Note that $F_i,G_i,H_i$ are quadratics in $\theta $, so the point $Q_\theta $ contains coordinates which are 
degree at most five in $\theta $. This  corresponds to the fact the $\V^3_2\cap \V^3_2=\V^9_1$, and as $\ell$ is a double line in both $\K_1$ and $\K_2$, it counts four times in the intersection, leaving a quintic $\V^5_1$ which contains the point $(0,1,0,0)$.

Finally, we want to find points of the form $Q_\theta $ which lie on $\K_3$. Substituting $Q_\theta $ into the equation $g_3=0$ gives $$(G_2H_1-G_1H_2)^2\, \theta  \Big( (H_2G_1-H_1G_2)F_3+(H_1F_2-H_2F_1)G_3+
(F_1G_2-F_2G_1) H_3\Big)=0. $$
This is an equation of degree 15 in $\theta$ (recall $\V^9_1\cap\V^3_2=\V^{15}_1$).
 Recall that $(G_2H_1-G_1H_2)\neq0$ for points on $\K_1\cap\K_2$, so eight of the roots do not contribute to a point in the general intersection. If $\theta =0$, then we have a point on the plane $y_0=0$, which lies in the kernel of the map $\sigma$, so is not of interest to us.
The remaining term is degree six in $\theta $, so there are in general six solutions. That is, excluding points with $y_0=0$, our three cubic surfaces generically meet in six points.
That is, the number of points of $\SC$ in a generic 5-space is six, and so the variety $\V(\SC)$ has order six as required.
\end{proof}

\section{Coordinates in the Bose representation}\Label{sec:coord}

\subsection{Coordinates for a Bose plane $\Bo{X}$}\Label{sec:coord2}

Let $\tau$ be a primitive element of $\Fq$ with primitive polynomial $$x^3-t_2x^2-t_1x-t_0$$  for $t_0,t_1,t_2\in\Fq$. 
 Let $\bar X=(x,y,z)\in\PG(2,q^3)$.   
  Using homogeneous coordinates we can write $\bar X= \rho(x,y,z)=(\rho x,\rho y,\rho z)$ for any $\rho\in\Fqqq\setminus\{0\}$. 
  As $\rho$ varies,  we generate a related point of $\PG(8,q)$, giving us the $(q^3-1)/(q-1)$ points of the plane $\Bo{X}$. 
To describe this plane, we determine the coordinates of three non-collinear points in it. 

Firstly, consider the representation of the coordinates of $\bar X$ with $\rho=1$. 
 We can write $$\bar X=(x,y,z)=\big(x_0+x_1\tau+x_2\tau^2,\ y_0+y_1\tau+y_2\tau^2,\ z_0+z_1\tau+z_2\tau^2\big)$$ for unique $x_i,y_i,z_i\in\Fq$.  Related to this representation of $\bar X$  is the  point $X_0$ of $\PG(8,q)$ where
\begin{eqnarray}\label{eqnx0}
X_0=(x_0,\,x_1,\,x_2,\,y_0,\,y_1,\,y_2,\,z_0,\,z_1,\,z_2).
\end{eqnarray} 
Secondly, consider the representation  of the coordinates of $\bar X$ with $\rho=\tau$, so $\bar X=\tau(x,y,z)=(\tau x,\tau y,\tau z)$. The first coordinate expands as  
\begin{eqnarray*}
\tau x&=&\tau(x_0+x_1\tau+x_2\tau^2)
= x_2t_0+(x_0+x_2t_1)\tau+(x_1+x_2t_2)\tau^2.
\end{eqnarray*} 
Similarly, we can expand the second and third coordinates $\tau y$ and $\tau z$. 
Corresponding  to this representation of $\bar X$  is the  point $X_1\in\PG(8,q)$ whose first three coordinates are 
\begin{eqnarray}\label{eqnx1}
X_1=\Big(x_2t_0,\ x_0+x_2t_1,\ x_1+x_2t_2,\ \ \ldots\ldots\Big).
\end{eqnarray}  
 Thirdly, consider the representation $\bar X=\tau^2(x,y,z)$. The first coordinate is $\tau^2 x=\tau^2(x_0+x_1\tau+x_2\tau^2)$. Expanding and simplifying yields 
$$ t_0(x_1+t_2x_2)\,+\,\big(t_0x_2+t_1(x_1+t_2x_2)\big)\tau\,
+\,\big(x_0+t_1x_2+t_2(x_1+t_2x_2)\big)\tau^2.
$$
Similarly, we expand the other two  coordinates $\tau^2 y$ and $\tau^2 z$.
 Corresponding  to this representation of $\bar X$  is the  point $X_2\in\PG(8,q)$ whose first three coordinates are 
 \begin{eqnarray}\label{eqnx2}
 X_2=\Big(t_0(x_1+t_2x_2),\ t_0x_2+t_1(x_1+t_2x_2),\ x_0+t_1x_2+t_2(x_1+t_2x_2),\ \ \ldots\ldots\Big). 
\end{eqnarray}
More generally, consider the representation $\bar X=\rho(x,y,z)=(\rho x,\rho y,\rho z)$, where $\rho=p_0+p_1\tau+p_2\tau^2$ for unique $p_0,p_1,p_2\in\Fq$. Straightforward expanding  and simplifying yields the   point of $\PG(8,q)$ corresponding to this representation of $\bar X$ is the point $p_0X_0+p_1X_1+p_2X_2$ (where $X_0,X_1,X_2$ are given in (\ref{eqnx0}), (\ref{eqnx1}),  (\ref{eqnx2})).  We have proved the following result.

\begin{lemma}\Label{lem:coordX}
Let $\bar X$ be a point in $\PG(2,q^3)$, then 
 the Bose representation of $\bar X$ in $\PG(8,q)$ is the plane $$\Bo{X}=\langle X_0,X_1,X_2\rangle$$
 where $X_0,X_1,X_2$ are given in (\ref{eqnx0}), (\ref{eqnx1}) and (\ref{eqnx2}).
 \end{lemma}

\subsection{The transversal planes of the Bose 2-spread $\mathbb S$}\Label{sec-ai}

We determine the coordinates of the three transversal planes of the Bose 2-spread $\mathbb S$. First, we define three points in $\PG(8,q^3)$, we state this as a definition for easy reference. 

\begin{definition}\label{def-aA}
Denote the following constants ${\mathsf a}_0,{\mathsf a}_1,{\mathsf a}_2\in\Fqqq$ and the points $A_0,A_1,A_2\in\PG(8,q^3)$, as 
\begin{eqnarray*}
{\mathsf a}_0&=&t_1+t_2\tau-\tau^2=-\tau^q\tau^{q^2},\\ {\mathsf a}_1&=&t_2-\tau=\tau^q+\tau^{q^2},\\ {\mathsf a}_2&=&-1,
\\A_0&=&({\mathsf a}_0,{\mathsf a}_1,{\mathsf a}_2,\ 0,0,0,\ 0,0,0),\\ 
A_1&=&(0,0,0,\ {\mathsf a}_0,{\mathsf a}_1,{\mathsf a}_2,\ 0,0,0),\\
A_2&=&(0,0,0,\ 0,0,0,\ {\mathsf a}_0,{\mathsf a}_1,{\mathsf a}_2).
\end{eqnarray*}
\end{definition}

\begin{lemma}\Label{lem:tp} Using the notation of Definition~\ref{def-aA}, 
the three transversal planes of the Bose 2-spread $\mathbb S$ are  \[
\Pit=\langle A_0,A_1,A_2\rangle,\quad \Pit^q=\langle A_0^q,A_1^q,A_2^q\rangle,\quad \Pit^{q^2}=\langle A_0^{q^2},A_1^{q^2},A_2^{q^2}\rangle.
\]
Moreover, the point $\bar X=(x,y,z)\in\PG(2,q^3)$  corresponds to the  point $X$  in $\Pit$ where $X=\Bo{X}\star\cap\Gamma$ and 
\[
X={\mathsf a}_0X_0+{\mathsf a}_1X_1+{\mathsf a}_2X_2=xA_0+yA_1+zA_2,
\]
 where $X_0,X_1,X_2$ are given in (\ref{eqnx0}), (\ref{eqnx1}), (\ref{eqnx2}).
\end{lemma}

\begin{proof} 
Let $\bar X=(x,y,z)\in\PG(2,q^3)$, so $\Bo{X}=\langle X_0,X_1,X_2\rangle$ as calculated above. Straightforward calculations   show that  in $\PG(8,q^3)$ we have  ${\mathsf a}_0X_0+{\mathsf a}_1X_1+{\mathsf a}_2X_2=xA_0+yA_1+zA_2.$ Hence these are the homogeneous coordinates of a point that lies in the plane $\Bo{X}\star$ and in the plane $\langle A_0,A_1,A_2\rangle$. Hence the extension of the Bose spread element $\Bo{X}$ meets the plane $\langle A_0,A_1,A_2\rangle$. As every extended Bose spread element meets  the plane $\langle A_0,A_1,A_2\rangle$, it is one of the transversal planes, which we denote by $\Pit$. The other two transversal planes are hence $\Pit^q$ and $\Pit^{q^2}$.
\end{proof}

Note that the points $A_0,A_1,A_2$ of the transversal plane $\Pit$ correspond to the fundamental triangle of $\PG(2,q^2)$, namely $\bar A_0=(1,0,0)$, $\bar A_1=(0,1,0)$, $\bar A_2=(0,0,1)$. 

We can now write $\Bo{X}$ in terms of the point $X\in\Gamma$, namely 
$$\Bo{X}\star=\langle X,X^q,X^{q^2}\rangle,$$
 where
\[
\begin{array}{ccccccccc}
X&=&xA_0&+&yA_1&+&zA_2&\in&\Pit\\
X^{q}&=&x^qA_0^{q}&+&y^qA_1^{q}&+&z^qA_2^{q}&\in&\Pit^q \\
X^{q^2}&=&x^{q^2}A_0^{q^2}&+&y^{q^2}A_1^{q^2}&+&z^{q^2}A_2^{q^2}&\in&\Pit^{q^2}.
\end{array}
\]

In order to study $\Fq$-subplanes later in this article, we need to develop a description of the plane in $\PG(8,q^3)$ which contains the three points 
$xA_0+yA_1+zA_2$, $ xA_0^{q}+yA_1^{q}+zA_2^{q}$ and  $ xA_0^{q^2}+yA_1^{q^2}+zA_2^{q^2}$.
The next two subsections are devoted to carefully calculating a useful description of  these points.

\subsection{Conjugacy with respect to an  $\Fq$-subplane}\Label{sec-conj}

Let $\Fq$ denote the finite field of order $q$. 
An $\Fq$-subplane  of $\PG(2,q^3)$ is a subplane   of $\PG(2,q^3)$  which has order $q$, 
that is, a subplane isomorphic to $\PG(2,q)$. An $\Fq$-subline is a line of an $\Fq$-subplane, that is,  isomorphic to  $\PG(1,q)$. We will define conjugacy with respect to an $\Fq$-subplane or subline. 

First consider the $\Fq$-subplane $\bar \pi_0=\PG(2,q)$ of $\PG(2,q^3)$. There are two collineations in $\PGammaL(3,q^3) $ which have order 3 and fix $\bar\pi_0$ pointwise, namely $\bar{\mathsf c}$ and $\bar{\mathsf c}^2$ where 
 \begin{eqnarray*}
 \bar{\mathsf c} \colon \PG(2,q^3)&\longrightarrow& \PG(2,q^3)  \\
\bar X= \begin{pmatrix} x\\y\\z\end{pmatrix}&\longmapsto & \bar X^q= \begin{pmatrix} x^q\\y^q\\z^q\end{pmatrix}.
 \end{eqnarray*}
  In fact, ${\bar G}_{\pi_0}=\langle\bar{\mathsf c} \rangle $ is the unique subgroup of $\PGammaL(3,q^3)$ which fixes $\bar \pi_0$ pointwise and has order 3. For $\bar X\in\PG(2,q^3)\setminus\bar \pi_0$, the three points $\bar X$, $\bar X^{\bar{\mathsf c}}$, $\bar X^{\bar{\mathsf c}^2}$ are called \emph{conjugate with respect to $\bar \pi_0$}. 
The collineation $\bar{\mathsf c} $ can be extended to act on points of $\PG(2,q^6)$, and we denote the extension of $\bar{\mathsf c}$ to $\PGammaL(3,q^6)$ by  $\bar{\mathsf c} $ as well. That is, acting on  points of $\PG(2,q^6)$, we have 
 $\bar{\mathsf c} (\bar X)=\bar X^q$, so $\bar{\mathsf c}$ has order 6 when acting on $\PG(2,q^6)$.  Under the collineation $\bar{\mathsf c}$, a point $\bar X\in\PG(2,q^6)$ lies in an orbit of size:  
 $1$ if  $\bar X\in\bar \pi_0$;
 $3$ if $\bar X\in\PG(2,q^3)\setminus\bar \pi_0$;
 $2$ or $6$ if $\bar X\in\PG(2,q^6)\setminus\PG(2,q^3)$, depending on whether   $\bar X$ belongs to the $\Fqq$-subplane
 $\PG(2,q^2)$ that contains $\bar \pi_0=\PG(2,q)$, or not.

More generally, let $\bar \pi$ be an $\Fq$-subplane of $\PG(2,q^3)$.
 Acting on the points of $\PG(2,q^3)$ is a unique collineation group ${\bar G}_\pi\subseteq \PGammaL(3,q^3)$ which fixes $\bar \pi$ pointwise and has order 3. 
 We wish to distinguish between the two non-identity collineations in ${\bar G}_\pi$, and do so as follows. 
 Consider any homography that maps 
 $\bar\pi$ to $\bar\pi_0$, and denote its  $3\times 3$ non-singular matrix over $\Fqqq$   by $A$, so if $\bar X\in\bar  \pi$, then $A\bar X\in \bar \pi_0$. 
 Let  $\ocpi (\bar X)= A^{-1} \,\bar { \mathsf c}( A \bar X).$
As $\ocpi$ has order 3 and fixes $\bar \pi$ pointwise, 
 we have ${\bar G}_\pi=\langle \ocpi \rangle $. We expand $\ocpi$, 
 and to avoid confusion use the following notation. For a $3\times 3$ matrix  $A=(a_{ij})$, $i,j=1,2,3$,  we let the matrix $A^\sigma=(a_{ij}^q)$, $i,j=1,2,3$. 
Thus  $ \ocpi (\bar X)=A^{-1}A^\sigma \bar X^q$, or  writing $B= A^{-1}A^\sigma$, we have $ \ocpi (\bar X)=B \bar X^q$. That is,  
we can without loss of generality  write ${\bar G}_\pi=\langle \ocpi \rangle $ with 
 \begin{eqnarray}\label{defcpi}
 \ocpi\colon \PG(2,q^3)&\longrightarrow& \PG(2,q^3) \nonumber \\
 \bar  X&\longmapsto & B \bar X^q   \end{eqnarray}
 with $B$ a $3\times3$ non-singular matrix over $\Fqqq$.   
For $\bar X\in\PG(2,q^3)\setminus\bar \pi$, the three  points $\bar X$, $\bar X^{\ocpi}$, $\bar X^{\ocpis}$ are called \emph{conjugate with respect to $\bar \pi$}. Note that  
 $\bar X$, $\bar X^{\ocpi}$, $\bar X^{\ocpis}$ are collinear  if and only if $\bar X$ lies on an extended line of $\bar \pi$.

 Suppose we extend the plane $\PG(2,q^3)$ to $\PG(2,q^6)$. 
Then the collineation $\ocpi\in \PGammaL(3,q^3)$ has a natural extension to a collineation of $\PGammaL(3,q^6)$
acting on points of $\PG(2,q^6)$. We also denote the extended collineation by $\ocpi$, so
\begin{eqnarray*}
 \ocpi\colon \PG(2,q^6)&\longrightarrow& \PG(2,q^6)\\
\bar  X&\longmapsto & B \bar X^q. \end{eqnarray*}
 The collineation $\ocpi$ has order 3 when acting on $\PG(2,q^3)$, and order 6 when acting on $\PG(2,q^6)$. 
Under the collineation $\ocpi$, a point $\bar X\in\PG(2,q^6)$ lies in an orbit of size:
 $1$ if  $\bar X\in\bar \pi$;
 $3$ if $\bar X\in\PG(2,q^3)\setminus\bar \pi$;
 $2$ or $6$ if $\bar X\in\PG(2,q^6)\setminus\PG(2,q^3)$.

 Similarly, if $\bar b$ is an $\Fq$-subline of a line $\bar \ell_b$ of $\PG(2,q^3)$, then acting on the points of $\bar \ell_b$ is a unique collineation group $\bar G_b\subseteq \PGammaL(2,q^3)$ of order 3 which fixes $\bar b$ pointwise. 
Moreover, ${\bar G}_\pi$ restricted to acting on $\bar \ell_b$ is isomorphic to $\bar G_b$ if and only if  $\bar b$ is a line of $\bar \pi$. 
Without loss of generality we can write ${\bar G}_b=\langle \ocb \rangle $  where for a point $\bar X\in\bar\ell_b$, 
\begin{eqnarray}
\label{defcb}
\ocb(\bar X)&=&D\bar X^q
\end{eqnarray}
 with $D$ a non-singular matrix over $\Fqqq$.

\subsection{Conjugacy with respect to the  $\Fq$-subplane $\bar\pi_0=\PG(2,q)$ }

We now return to the $\Fq$-subplane $\bar\pi_0=\PG(2,q)$ of $\PG(2,q^3)$ and look in more detail at the collineation $\bar{\mathsf c}\in\PGammaL(3,q^3) $ defined by 
 \begin{eqnarray*}
 \bar{\mathsf c} \colon \PG(2,q^3)&\longrightarrow& \PG(2,q^3)  \\
\bar X= (x,y,z)&\longmapsto &\bar X^{\mathsf c}= \bar X^q= (x^q,y^q,z^q)
 \end{eqnarray*}
We have    ${\bar {G}}_{\pi_0}=\{id,{\bar {\mathsf c}},{\bar {\mathsf c}}^2\}$ where for  $\bar X=(x,y,z)\in \PG(2,q^3)$, $$\bar{X}^{\bar {\mathsf c}}=(x^q,y^q,z^q)\quad \textup{ and }\quad \bar{X}^{\bar {\mathsf c}^2}=(x^{q^2},y^{q^2},z^{q^2}).$$
The collineation ${\bar {\mathsf c}}$ induces a map denoted $\mathsf c$ which acts on the transversal plane $\Pit$ in  the Bose representation. In $\PG(8,q^3)$, $\pi$ is an $\Fq$-subplane of the transversal plane $\Gamma$, and ${\mathsf c}$ acts on the points of $\Gamma$, note that ${\mathsf c}$ is not a collineation of $\PG(8,q^3)$. By Lemma~\ref{lem:tp}, 
a point  $X\in\Gamma$ has form 
$$\begin{array}{rclclcl}
X&=&x A_0+yA_1+z A_2&\in&\Gamma
\end{array}$$ 
for some $x,y,z\in\Fqqq$. So 
$$\begin{array}{rclcrcrcrcl}
X^{\mathsf c}&=&x^qA_0&+&y^qA_1&+&z^qA_2&\in&\Pit,\\
X^{\mathsf c^2}&=&x^{q^2}A_0&+&y^{q^2}A_1&+&z^{q^2}A_2&\in&\Pit.
\end{array}
$$
Moreover, we will be interested in the images of these points under the map $X\mapsto X^q$ of $\PG(8,q^3)$, in particular, 
$$\begin{array}{rclclclclcl}
(X^{\mathsf c^2})^{q}&=&xA_0^{q}&+&yA_1^{q}&+&zA_2^{q}&\in&\Pit^q,\\
(X^{\mathsf c})^{q^2}&=&xA_0^{q^2}&+&yA_1^{q^2}&+&zA_2^{q^2}&\in&\Pit^{q^2},\\
\end{array}
$$

For  $X=xA_0+yA_1+zA_2\in\Pit$, we will be interested in the plane 
\begin{eqnarray}
\langle \, xA_0+yA_1+zA_2,\ xA_0^{q}+yA_1^{q}+zA_2^{q},\  xA_0^{q^2}+yA_1^{q^2}+zA_2^{q^2}\rangle=\langle \,X,\ (X^{\mathsf c^2})^{q},\ (X^{\mathsf c})^{q^2}\,\rangle.\label{eqn-1}
\end{eqnarray}

Consider the extension to $\PG(8,q^6)$. 
Let ${\mathsf e}\in\PGammaL(9,q^6)$ be the unique involution acting on points of $\PG(8,q^6)$ fixing $\PG(8,q^3)$ pointwise. We say the points $X, X^{\mathsf e}$ are conjugate with respect to the quadratic extension from $\PG(8,q^3)$ to $\PG(8,q^6)$, and have 
$$X=(x_0,\ldots,x_{8})\quad \longmapsto \quad
X^{\mathsf e}=(x_0^{q^3},\ldots,x_{8}^{q^3}).$$ 
In particular, for $x,y,z\in\Fqqqqqq$, not all zero, we have  $$(xA_0+yA_1+zA_2)^{\mathsf e}= x^{q^3} A_0+y^{q^3} A_1+z^{q^3} A_2.$$
As $\mathsf e$ fixes $\PG(8,q^3)$ pointwise, it fixes $\Pit$ pointwise, hence $\mathsf e$ induces an automorphism of the plane $\Pit$. Moreover, for  a point $X\in\Pit\blackstar\setminus\Pit$, we have 
$
X^{\mathsf e}=X^{q^3}=X^{\mathsf c^3}$.
It is straightforward to verify that 
$$\begin{array}{rclclclclcl}
(X^{\mathsf c^2\mathsf e})^q&=&(X^{\mathsf c^5})^q&=& xA_0^q&+&yA_1^q&+&zA_2^q &\in&\Pit^q\\
(X^{\mathsf c\mathsf e})^{q^2}&=&(X^{\mathsf c^4})^{q^2}&=&x A_0^{q^2}&+&y A_1^{q^2}&+&z A_2^{q^2}&\in&\Pit^{q^2}.
\end{array}$$
For $X=xA_0+yA_1+zA_2\in\Pit\blackstar$, we will be interested in the plane 
\begin{eqnarray}
\langle \, xA_0+yA_1+zA_2,\ xA_0^{q}+yA_1^{q}+zA_2^{q},\  xA_0^{q^2}+yA_1^{q^2}+zA_2^{q^2}\rangle=\langle \,X,\ (X^{\mathsf c^2\mathsf e})^{q},\ (X^{\mathsf c\mathsf e})^{q^2}\,\rangle
.\label{eqn-2}
\end{eqnarray}
Note that if $X\in\Pit$, then this equation agrees with (\ref{eqn-1}).

\subsection{Coordinates in the Bruck-Bose representation }

We can construct the Bruck-Bose representation of $\PG(2,q^3)$ in $\PG(6,q)$ by intersecting the Bose representation with a 6-space $\Sigma_{6,q}$ of $\PG(8,q)$ which contains a unique 5-space that meets $\mathbb S$ in $q^3+1$ planes. To obtain the same coordinates  for the Bruck-Bose representation used in  \cite[Section 2.2]{BJ-FFA}, we take the line at infinity in $\PG(2,q^3)$ to have equation $z=0$, which contains the points $\bar X=(1,0,0),\bar Y=(0,1,0)$. In $\PG(8,q^3)$, this  corresponds to the line  $g$ in the transversal plane $\Pit$ where  $g=\langle A_0,A_1\rangle$. Take $\Sigma_{6,q}$ as the 6-space of $\PG(8,q)$ consisting of the points $(x_0,x_1,x_2,x_3,x_4,x_5,x_6,0,0)$, $x_i\in\Fq$, not all 0. Then $\sistar=\langle g,g^q,g^{q^2}\rangle$ and contains all points of form $(x_0,x_1,x_2,x_3,x_4,x_5,0,0,0)$. Further,   $g,g^q,g^{q^2}$ are the transversal lines of the regular 2-spread $\S$ in $\si$. An affine point $\bar P=(x,y,1)=(x_0+x_1\tau+x_2\tau^2, \ y_0+y_1\tau+y_2\tau^2, \ 1)$ of $\PG(2,q^3)$ corresponds to the affine point of $\Sigma_{6,q}\setminus\si$ denoted $[P]=\Bo{P}\cap\Sigma_{6,q}$, with coordinates $[P]=(x_0,x_1,x_2,\ y_0,y_1,y_2,\ 1,0,0)$.

\section{Varieties in the Bose representation}\Label{sec:var}

In this section we use coordinates to show how to convert a variety of $\PG(2,q^3)$ into a variety in the Bose representation $\PG(8,q)$. This generalises known techniques from the Bose representation of $\PG(2,q^2)$ in $\PG(5,q)$, and can be used more generally to convert a primal of $\PG(2,q^h)$ to  the intersection of varieties in the Bose representation $\PG(3h-1,q)$. Of particular interest is the variety-extensions to $\PG(8,q^3)$ and $\PG(8,q^6)$.  

We use the following notation. Let  $f(x_0,\ldots,x_n)$ be a homogeneous form over $\Fqh$ of degree $d$ in indeterminates $x_0,\ldots,x_n$. Then $V(f)$ denotes the set of points $X\in\PG(n,q^h)$ such that $f(X)=0$. 
If $f_1,\ldots,f_k$ are homogeneous forms over $\Fqh$ in indeterminates $x_0,\ldots,x_n$, then $V(f_1,\ldots,f_k)=V(f_1)\cap \cdots\cap V(f_k)$ denotes the set of points $X\in\PG(n,q^h)$ such that $f_i(X)=0$, $i=1,\ldots,k$.

\subsection{Varieties of $\PG(2,q^3)$}

We study varieties of $\PG(2,q^3)$ and their corresponding structure in the Bose representation. This study proceeds as follows. 
We first look in Lemma~\ref{lem:var1} at a homogeneous equation in $\PG(2,q^3)$ and convert it to a homogeneous equation in $\PG(8,q)$. Next, in Lemma~\ref{lem:var2}, we look at a variety $\overline \K$ in $\PG(2,q^3)$, and use the calculations from Lemma~\ref{lem:var1} to show that the corresponding set of points  in  the Bose representation in $\PG(8,q)$, forms a variety denoted $\V(\Bo{\K})$. 
We determine the extension $\V(\Bo{\K})\star$ to $\PG(8,q^3)$ in Lemma~\ref{lem:var3}. Lemma~\ref{lem:var3.2} is a stepping stone to  Lemma~\ref{lem:var4} which gives a geometric description of  $\V(\Bo{\K})\star$. The results of these lemmas is then summarised in Theorem~\ref{lem:cone}. 
In Section~\ref{conic-sec} we then use this result to study conics of $\PG(2,q^3)$ in the Bose representation.

\begin{lemma}\Label{lem:var1}
Let $\bar F(x,y,z)$ be a homogeneous equation of degree $k$ over $\Fqqq$. Using $\tau$ as in Section~\ref{sec:coord2}, we write  $x=x_0+\tau x_1+\tau^2 x_2$, $y=y_0+\tau y_1+\tau^2 y_2$, $z=z_0+\tau z_1+\tau^2 z_2$ for unique $x_i,y_i,z_i\in\Fq$. Expanding and simplifying yields $$\bar F(x,y,z)=G(x_0,x_1,x_2,y_0,y_1,y_2,z_0,z_1,z_2)= f_0+\tau f_1+\tau^2 f_2$$ where $f_i=f_i(x_0,x_1,x_2,y_0,y_1,y_2,z_0,z_1,z_2)$ is a homogeneous equation of degree $k$ over $\Fq$, $i=0,1,2$. 
\end{lemma}

\begin{proof}
Recall from Section~\ref{sec:coord2}, $\tau$  is a primitive element in $\Fqqq$ satisfying $\tau^3= t_0+t_1\tau+t_2\tau^2$ for some $t_0,t_1,t_2\in\Fq$.  
 Let $\bar F=\bar F(x,y,z)$ be a homogeneous form over $\Fqqq$ of degree $k$ in indeterminates $x,y,z$.
 For indeterminates $x_0,x_1,x_2,y_0,y_1,y_{2},z_0,z_1,z_{2}$, substitute $x=x_0+x_1\tau+ x_{2}\tau^{2}$, $y=y_0+y_1\tau+y_{2}\tau^{2}$, $z=z_0+z_1\tau+ z_{2}\tau^{2}$ in $\bar F$ to obtain the homogeneous form $G(x_0,x_1,x_2,y_0,y_1,y_{2},z_0,z_1,z_{2})$ of degree $k$. For each coefficient $a \in \Fqqq$ of $G$, rewrite as $a=a_0+a_1\tau  +a_{2}\tau^{2}$ for unique $a_0,a_1,a_{2} \in \Fq$. Then using $\tau^3= t_0+t_1\tau+t_2\tau^2$, we may write $$G=f_0+f_1\tau  +f_{2}\tau^{2}$$ for unique homogeneous forms $f_0,f_1,f_{2}$ over $\Fq$ of degree $k$ in indeterminates $x_0,x_1,x_2,y_0,y_1,y_{2},z_0,z_1,z_{2}$.
\end{proof}

Note that $G$ is a homogeneous equation with coefficients in $\Fqqq$, so it does not make sense to talk about a variety in $\PG(8,q)$ corresponding to $G$. However, $f_0$, $f_1$ and $f_2$ are homogeneous equations with coefficients in $\Fq$, so we have the corresponding varieties $V(f_0)$, $V(f_1)$  and $V(f_2)$ in $\PG(8,q)$. Moreover, a point $P\in\PG(8,q)$ satisfies $G(P)=0$ if and only if $f_0(P)=f_1(P)=f_2(P)=0$ if and only if $P\in V(f_0)\cap V(f_1)\cap V(f_2)=V(f_0,f_1,f_2)$.

\begin{lemma}\Label{lem:var2}
Let $\bar F(x,y,z)$ be a homogeneous equation of degree $k$ over $\Fqqq$, and let $\bar\K=V(\bar F)$ be the corresponding variety in $\PG(2,q^3)$. Let $G$, $f_0$, $f_1$ and $f_2$   be  derived from $\bar F$ as in Lemma~\ref{lem:var1}. 
 In  the Bose representation of $\PG(2,q^3)$ in $\PG(8,q)$, the pointset of $\Bo{\K}$ 
coincides with the pointset of the  variety $V(f_0,f_1,f_2)=
V(f_0)\cap V(f_1)\cap V(f_2)$. We define 
  $\V(\Bo{\K})=
V(f_0,f_1,f_2)$.
\end{lemma}

\begin{proof}
Using the coordinates described in Section~\ref{sec:coord2},  a  point $Q$ in $\PG(8,q)$ 
 corresponds to a unique point of $\PG(2,q^3)$ which we denote by $\bar X_Q$ as follows. Let $Q$ 
 have homogeneous coordinates
\begin{eqnarray*}
Q&=&(a_0,a_1,a_{2},b_0,b_1,b_{2},c_0,c_1,c_{2})
\\&\equiv&( \lambda a_0, \lambda a_1, \lambda a_{2}, \lambda b_0, \lambda b_1, \lambda b_{2}, \lambda c_0, \lambda c_1, \lambda c_{2})
\end{eqnarray*} 
for any $\lambda \in \Fq\setminus\{0\}$. 
The corresponding point in $\PG(2,q^3)$ has coordinates 
\begin{eqnarray*}
\bar X_Q&=& (\lambda a_0+\lambda a_1\tau  +\lambda a_{2}\tau^{2},\ \lambda b_0+\lambda b_1\tau  +\lambda b_{2}\tau^{2},\ \lambda c_0+\lambda c_1\tau  +\lambda c_{2}\tau^{2})\\
&\equiv&(a_0+a_1\tau  +a_{2}\tau^{2},\ b_0+b_1\tau  +b_{2}\tau^{2},\ c_0+c_1\tau  +c_{2}\tau^{2}).
\end{eqnarray*}
 By Lemma~\ref{lem:var1}, 
 $G(Q) =\bar F(\bar X_Q)$, so $G(Q)=0$ if and only if $\bar F(\bar X_Q)=0.$
 As noted above, $G(Q)=0$ if and only if $Q\in V(f_0,f_1,f_2)$. 
 Hence a  point $Q\in\PG(8,q)$ lies in  $V(f_0,f_1,f_2)$ if and only if the corresponding point $\bar X_Q$ lies in $V(\bar F)=\bar \K$.
 
 We now consider the converse, a 
 point $\bar P$ in $\PG(2,q^3)$ 
 corresponds to $q^2+q+1$ points  in $\PG(8,q)$, namely the $q^2+q+1$ points of the plane $\Bo{P}$. We want to show that if $\bar P\in\bar\K$, then in $\PG(8,q)$, every point $Y\in \Bo{P}$ lies in 
 the variety $V(f_0,f_1,f_2)$. 
 The point $\bar P$ has homogeneous coordinates $\bar P=(x,y,z)\equiv \rho (x,y,z)$, for any $\rho\in\Fqqq\setminus\{0\}$. 
 By Lemma~\ref{lem:coordX}, we have 
   $\Bo{P}=\langle X_0,X_1,X_2\rangle$ with points $X_0,X_1,X_2$ defined in Section~\ref{sec:coord2}. Moreover, writing $\rho=p_0+p_1 \tau+p_2\tau^2$, for $p_0,p_1,p_2\in\Fq$, then by  Section~\ref{sec:coord2}, the point in $\Bo{P}$ corresponding to the coordinates   $\rho(x,y,z)$ is the point $$Y_\rho=p_0 X_0+p_1 X_1+p_2 X_2\ \in \Bo{P}.$$ 
  So $\bar P\in\bar K$ if and only if $\bar F(\bar P)=\bar F\big(\rho(x,y,z)\big)=0$ for all $\rho\in\Fqqq\setminus\{0\}$ if and only if 
  $G(Y_\rho)=0$  for all $\rho\in\Fqqq\setminus\{0\}$  if and only if $Y_\rho\in V(f_0,f_1,f_2)$  for all $\rho\in\Fqqq\setminus\{0\}$. Note that the points $Y_\rho$, $\rho\in\Fqqq\setminus\{0\}$ are exactly the points of the plane $\Bo{P}$. Hence $\bar P\in\bar \K$ if and only if $\Bo{P}\subseteq V(f_0,f_1,f_2)$. 
 It follows that $\Bo{\K}=V(f_0,f_1,f_2)$ as required. 
\end{proof}

As noted above, $G$ is a homogeneous equation with coefficients in $\Fqqq$, so $G$ corresponds to a variety in $\PG(8,q^3)$. That is, 
we let $V(G)$ denote the variety of $\PG(8,q^3)$ consisting of the points $X$ of $\PG(8,q^3)$ satisfying $G(X)=0$. 
Further, we can define the varieties $V(G^q)$ and $V(G^{q^2})$ in $\PG(8,q^3)$. 

\begin{lemma}\Label{lem:var3}
Let $\bar F(x,y,z)$ be a homogeneous equation of degree $k$ over $\Fqqq$, and let $\bar\K=V(\bar F)$ be the corresponding variety in $\PG(2,q^3)$. Let $G$, $f_0$, $f_1$ and $f_2$   be  derived from $\bar F$ as in Lemma~\ref{lem:var1}, and let $\V(\Bo{\K})=
V(f_0,f_1,f_2)$.
 The variety-extension of $\V(\Bo{\K})$ to $\PG(8,q^3)$ is $$\V(\Bo{\K})\star=
V(f_0)\star\cap V(f_1)\star\cap V(f_2)\star= V(G,G^q,G^{q^2}).$$ 
\end{lemma}

\begin{proof}
As $\V(\Bo{\K})=V(f_0)\cap V(f_1)\cap V(f_2)$, we have 
$$\V(\Bo{\K})\star=V(f_0)\star\cap V(f_1)\star\cap V(f_2)\star.$$  
The set of points of $\PG(8,q^3)$  satisfying the three equations $f_0=0$, $f_1=0$, $f_2=0$ is equivalent to the set of points satisfying any three linearly independent equations of form $\lambda_0f_0+\lambda_1f_1+\lambda_2 f_2$ where $\lambda_1,\lambda_2,\lambda_3\in\Fqqq$. 
Recalling that $f_0,f_1,f_2$ are equations over $\Fq$, we consider the three linearly independent equations $$G=f_0+\tau f_1+\tau^2 f_2,\quad G^q= f_0+\tau^q f_1+\tau^{2q} f_2,\quad G^{q^2}= f_0+\tau^{q^2} f_1+\tau^{2{q^2}} f_2.$$ So a point $P\in\PG(8,q^3)$ satisfies $f_0(P)=f_1(P)=f_2(P)=0$ if and only if $G(P)=G^q(P)=G^{q^2}(P)=0$. Hence 
$$\V(\Bo{\K})\star=
V(G)\cap V(G^{q})\cap V(G^{q^2})=V(G,G^{q},G^{q^2})$$ as required.
\end{proof}

\begin{lemma}\Label{lem:var3.2}
Let $\bar F(x,y,z)$ be a homogeneous equation of degree $k$ over $\Fqqq$, and let $\bar\K=V(\bar F)$ be the corresponding variety in $\PG(2,q^3)$. Let $G$, $f_0$, $f_1$ and $f_2$   be  derived from $\bar F$ as in Lemma~\ref{lem:var1}, and let $\V(\Bo{\K})=
V(f_0,f_1,f_2)$.
Consider the Bose representation of $\PG(2,q^3)$ with transversal planes $\Pit,\Pit^q,\Pit^{q^2}$ in $\PG(8,q^3)$. 
 Then the variety $V(G)$ of $\PG(8,q^3)$ is a cone with base $V(G) \cap\Pit=\K$  and vertex $\langle \Pit^q,\Pit^{q^2} \rangle$. 
\end{lemma}

\begin{proof}
We first determine how the variety $V(G)$ meets the transversal plane $\Pit$.   By Lemma~\ref{lem:tp}, 
a point $Q$ in the transversal plane $\Pit$ has coordinates 
$$Q=xA_0+yA_1+zA_2=(x\mathsf a_0,x\mathsf a_1,x\mathsf a_2,\ y\mathsf a_0,y\mathsf a_1,y\mathsf a_2,\ z\mathsf a_0,z\mathsf a_1,z\mathsf a_2)$$ for some $x,y,z\in\Fqqq$, with   
 $A_i$ and $\mathsf a_i$ as  in Definition~\ref{def-aA}. Moreover, the points of $\Pit$ are in one-to-one correspondence with the points of $\PG(2,q^3)$ and 
 $Q$ corresponds to the point $\bar Q$ of $\PG(2,q^3)$ where
$$\bar Q
= (x,y,z).$$
By Lemma~\ref{lem:var1}, $G(Q)=\bar F (\bar Q)$. So 
the point  $Q$ of $\Pit$ lies in the variety $ V(G)$ if and only if the point $\bar Q$ of $\PG(2,q^3)$ lies in $V(\bar F)=\bar \K$. 
That is, $Q\in V(G) \cap\Pit$ if and only if $\bar Q\in \bar \K$, and so  
 $V(G)\cap \Pit=\K$ as required. 
 
 We now determine the base of $V(G)$. 
 As $V(G)$ does not contain the plane $\Pit$, the maximum dimension of the singular space of $V(G)$ is five. We show that the 5-space 
$\langle \Pit^q,\Pit^{q^2} \rangle$ is the singular space of $V(G)$ by showing that every point of $V(G)$ lies on a  line joining a point $Q\in\K$ to a point   $R\in\langle \Pit^q,\Pit^{q^2} \rangle$. 
Let $Q\in\K$, $\ R\in \langle \Pit^q,\Pit^{q^2} \rangle$ and $P\in QR$. By Lemma~\ref{lem:tp}, $\Pit=\langle A_0,A_1,A_2\rangle$, $\Pit^q=\langle A_0^q,A_1^q,A_2^q\rangle$ and $\Pit^{q^2}=\langle A_0^{q^2},A_1^{q^2},A_2^{q^2}\rangle$. So $P$ has homogeneous coordinates of form 
 $$P=
xA_0+yA_1+zA_2+ rA_0^{q}+sA_1^{q}+tA_2^{q}+uA_0^{q^2}+vA_1^{q^2}+wA_2^{q^2},$$ for some  $x,y,z,r,s,t,u,v,w\in\Fqqq$.  
Simplifying $P$ (using the coordinates for $A_i$ and $\mathsf a_i$  as  in Definition~\ref{def-aA}) we calculate    the first three coordinates of $P$ are $$(x{\mathsf a}_0+r{\mathsf a}_0^q+u{\mathsf a}_0^{q^2},\ x{\mathsf a}_1+r{\mathsf a}_1^q+u{\mathsf a}_1^{q^2},\  x{\mathsf a}_2+r{\mathsf a}_2^q+u{\mathsf a}_2^{q^2}).$$ 
As in the proof of Lemma~\ref{lem:var2}, recall that $P$ corresponds to a unique point of $\PG(2,q^3)$ which we denote by $\bar  X_P$, and  the first coordinate of   $\bar X_P$   is $$x{\mathsf a}_0+r{\mathsf a}_0^q+u{\mathsf a}_0^{q^2}\, +\,  (x{\mathsf a}_1+r{\mathsf a}_1^q+u{\mathsf a}_1^{q^2})\tau\,+\,   (x{\mathsf a}_2+r{\mathsf a}_2^q+u{\mathsf a}_2^{q^2})\tau ^2.$$
Straightforward manipulation shows that $$\mathsf {\mathsf a}_0^q+\tau{\mathsf a}_1^q+\tau^2{\mathsf a}_2^q={\mathsf a}_0^{q^2}+\tau{\mathsf a}_1^{q^2}+\tau^2{\mathsf a}_2^{q^2}=0,$$ and using this we simplify the first coordinate of $\bar X_P$ to $
({\mathsf a}_0+\tau{\mathsf a}_1+\tau^2{\mathsf a}_2)\, x$. Similarly we calculate the other coordinates of $P$,  and  the coordinates of $\bar X_P$ are $$\bar X_P = ({\mathsf a}_0+\tau{\mathsf a}_1+\tau^2{\mathsf a}_2)\,  \big(x,y,z\big) \, \equiv\, (x,y,z) = \bar Q.$$
By Lemma~\ref{lem:var1}, $G(P)=\bar F(\bar X_P)$, so
 $P$ lies in the variety $V(G)$ if and only if the point $\bar X_P=\bar Q$ lies in $\bar \K$. 
That is, if $P$ is on a line joining $Q\in\Pit$ with a point $R$ of $\langle \Pit^q,\Pit^{q^2} \rangle$, then
$G(P)=0$ if and only if $\bar F(\bar Q)=0$. That is, $P\in V(G)$ if and only if $\bar Q\in\bar \K$ if and only if $Q\in\K$. 
  Hence $V(G)$ is a cone with base $\K$  and vertex $\langle \Pit^q,\Pit^{q^2} \rangle$. 
\end{proof}

\begin{lemma}\Label{lem:var4}
Let $\bar F(x,y,z)$ be a homogeneous equation of degree $k$ over $\Fqqq$, and let $\bar\K=V(\bar F)$ be the corresponding variety in $\PG(2,q^3)$ and let $\V(\Bo{\K})$ be the corresponding variety in $\PG(8,q)$ as defined in Lemma~\ref{lem:var2}. Then in $\PG(8,q^3)$, the pointset of $\V(\Bo{\K})\star$ is equivalent to the pointset of the planes $\{\langle X,Y^q,Z^{q^2}\rangle\st X,Y,Z\in \K\}$.
\end{lemma}

\begin{proof}
By  Lemma~\ref{lem:var3.2}, $V(G)$ is a cone with base $V(G) \cap\Pit=\K$  and vertex $\langle \Pit^q,\Pit^{q^2} \rangle$. Hence 
 $V(G^q)$ is a cone with base $\K^{q^2}$ in $\Pit^{q^2}$ and vertex $\langle \Pit^q,\Pit \rangle$; and $V(G^{q^2})$ is a cone with base $\K^{q}$  and vertex $\langle \Pit,\Pit^{q^2} \rangle$. 
 Note that each of the three cones is a set of T-planes. 
 Hence the intersection of the three cones $V(G)$, $ V(G^{q})$, $  V(G^{q^2})$ is the set of points lying on the T-planes 
$$\{\langle X,Y^q,Z^{q^2}\rangle\st X,Y,Z\in \K\}$$ as required.
\end{proof}

Summarising the results of this section we have proved the following  Theorem.

\begin{theorem}\Label{lem:cone}
Let $\bar F(x,y,z)$ be a homogeneous equation of degree $k$ over $\Fqqq$, and let $\bar\K=V(\bar F)$ be the corresponding variety in $\PG(2,q^3)$. Using $\tau$ as in Section~\ref{sec:coord2}, we write  $x=x_0+\tau x_1+\tau^2 x_2$, $y=y_0+\tau y_1+\tau^2 y_2$, $z=z_0+\tau z_1+\tau^2 z_2$ for unique $x_i,y_i,z_i\in\Fq$. Expanding and simplifying yields $$\bar F(x,y,z)=G(x_0,x_1,x_2,y_0,y_1,y_2,z_0,z_1,z_2)= f_0+\tau f_1+\tau^2 f_2$$ where $f_i=f_i(x_0,x_1,x_2,y_0,y_1,y_2,z_0,z_1,z_2)$ is a homogeneous equation of degree $k$ over $\Fq$, $i=0,1,2$. 
Consider the Bose representation of $\PG(2,q^3)$. 
\begin{enumerate}
\item In $\PG(8,q)$, the pointset of $\Bo{\K}$ 
forms a variety  $V(f_0,f_1,f_2)$ which we denote  by $\V(\Bo{\K})$.
\item In $\PG(8,q^3)$, 
$$\V(\Bo{\K})\star=
 V(G,G^q,G^{q^2})$$ 
and the pointset of $\V(\Bo{\K})\star$ is equivalent to the pointset of the planes $\{\langle X,Y^q,Z^{q^2}\rangle\st X,Y,Z\in \K\}$.
 
\end{enumerate}
\end{theorem}

\subsection{A Bose representation convention regarding variety-extensions}

Using the notation of Theorem~\ref{lem:cone}, we have the pointset of the Bose representation of the variety $V(\bar F)$ of $\PG(2,q^3)$ is the variety $V(f_0,f_1,f_{2})$ of $\PG(8,q)$. Note that this pointset may be the pointset of more than one variety of $\PG(8,q)$. In particular, as $a^{q+r}=a^{1+r}$, for all $a \in \Fq$, where $r \ge 0$, if the form $f_i$ of degree $d$ only contains terms having an indeterminate raised to a power greater than or equal to $q$, then reducing these exponents by $q-1$ yields a homogeneous form $g_i$ of degree $d-q+1$ such that the evaluations of $f_i$ and $g_i$ over $\Fq$ are equal. Thus the pointsets of the varieties $V(f_0,f_1,f_{2})$ and $V(g_0,g_1,g_{2})$ will be equal in $\PG(8,q)$. Note, however, that these evaluations may not be equal over the extension $\Fqqq$ and consequently  these varieties may not yield the same pointset when extended to the extended projective space $\PG(8,q^3)$. 

As this is an important notion, we illustrate this with two examples in 
the well known Bose representation  of $\PG(2,q^2)$ in $\PG(5,q)$.

\emph{Example 1} Let $\bar \U$ be a classical unital of $\PG(2,q^2)$. So $\bar  \U$ is projectively equivalent to the variety $V(\bar F)$ where $\bar F(x,y,z)=x^{q+1}+y^{q+1}+z^{q+1}$. Let $\tau \in \Fqq$ have minimal polynomial $x^2-t_1x-t_0$. Let $x_0,x_1,y_0,y_1,z_0,z_1$ be indeterminates and substitute $x=x_0+x_1\tau,y=y_0+y_1\tau,z=z_0+z_1\tau$ in $\bar F$ to obtain the form $$G(x_0,x_1,y_0,y_1,z_0,z_1)=(x_0+x_1\tau)^{q+1}+(y_0+y_1\tau)^{q+1}+(z_0+z_1\tau)^{q+1}.$$
As $\tau^q=t_1-\tau$ and $\tau\tau^q=-t_0$, this becomes
$G(x_0,x_1,y_0,y_1,z_0,z_1)=x_0^{q+1}+\tau x_0^qx_1+(t_1-\tau)x_0x_1^q-t_0x_1^{q+1}+y_0^{q+1}+\tau y_0^qy_1+(t_1-\tau)y_0y_1^q)-t_0y_1^{q+1}
+z_0^{q+1}+\tau z_0^qz_1+(t_1-\tau)z_0z_1^q)-t_0z_1^{q+1}.$
It follows that $$f_0(x_0,x_1,y_0,y_1,z_0,z_1)=x_0^{q+1}+t_1x_0x_1^q-t_0x_1^{q+1}+y_0^{q+1}+t_1y_0y_1^q-t_0y_1^{q+1}+z_0^{q+1}+t_1z_0z_1^q-t_0z_1^{q+1}$$
and
$$f_1(x_0,x_1,y_0,y_1,z_0,z_1)=
-x_0x_1^q+x_0^qx_1-y_0y_1^q+y_0^qy_1-z_0z_1^q+z_0^qz_1.$$
We now consider the corresponding forms $g_i$ where the exponents of the forms $f_i$ are reduced by $q-1$, giving $$g_0(x_0,x_1,y_0,y_1,z_0,z_1)=x_0^2+t_1x_0x_1-t_0x_1^2+y_0^2-t_0y_1^2+z_0^2-t_0z_1^2$$
and
$$g_1(x_0,x_1,y_0,y_1,z_0,z_1)=0.$$
The forms $f_0,f_1$ have degree $q+1$, while  the form $g_0$ has degree $2$ and $g_1$ is identically zero. 
As we have two varieties $\Q=V(g_0,g_1)$ and $\V=V(f_0,f_1)$ covering the same pointset of $\PG(5,q)$, there are two variety-extensions $\Qstar$ and $\Vstar$ to $\PG(5,q^2)$.  

There are two possible ways to consider the extension. 
 The standard way used in the literature to extend (see \cite{UnitalBook}) is to consider the variety of $\PG(5,q)$  corresponding to a unital to be the elliptic quadric $\Q=V(g_0,g_1)=V(g_0)$.
 The  variety-extension of this elliptic quadric to $\PG(5,q^2)$ is a hyperbolic quadric that contains the transversal plane $\Pit$ of the Bose representation.
 Alternatively,  we can consider the variety of $\PG(5,q)$ corresponding to a unital to be the 
variety $\V=V(f_0,f_1)$. 
By Theorem~\ref{lem:cone}, the extension of the variety $\V=V(f_0,f_1)$ to $\PG(5,q^2)$ is a variety  which consists of the points on the lines $XY^q$ for points $X,Y \in \U$ (where $\U$ is the classical unital in the transversal plane $\Pit$ that corresponds to $\bar\U$). 
That is, the variety-extension  meets  the transversal plane $\Pit$ in a unital.

%
%

\emph{Example 2} 
Consider the Bose representation of the subplane $\bar\pi_0=\PG(2,q)$ of $\PG(2,q^2)$.
Let $\bar F_0(x,y,z)=xy^q-x^qy$, $\bar F_1(x,y,z)=yz^q-y^qz$, $\bar F_2(x,y,z)=zx^q-z^qx$, then $\bar\pi_0$  may be described as the variety $V(\bar F_0,\bar F_1,\bar F_2)$. A similar analysis to Example 1 yields varieties $V(f_{i,j})$, $i=0,1,2$, $j=0,1$, with 
$f_{0,0}=x_0y_0^q-x_0^qy_0 +t_1(x_0y_1^q-x_1y^qy_0)-t_0(x_1y_1^q-x_1^qy_1)$, 
$f_{0,1}=x_0y_1^q+x_1^qy_0+x_1y_0^q-x_0^qy_1$, and so on. Hence reducing the exponents gives varieties  $V(g_{i,j})$, $i=0,1,2$, $j=0,1$, with 
$g_{0,0}=
t_1(x_0y_1-x_1y_0)$, $g_{0,1}=-2(x_0y_1-x_1y_0)$, and so on. 
 The forms defining the variety $\K=V(f_{0,0},f_{0,1},f_{1,0},f_{1,1},f_{2,0},f_{2,1})$ have degree $q+1$, and the forms defining the variety $\K'=V(g_{0,0},g_{0,1},g_{1,0},g_{1,1},g_{2,0},g_{2,1})$  are quadrics. The two varieties $\K,\K'$   have identical pointsets in $\PG(5,q)$, but their extensions to $\PG(5,q^2)$ differ. The extension of $\K$ to $\PG(5,q^2)$ meets the transversal plane $\Pit$ in the Baer subplane $\pi_0$, but the extension of $\K'$ (a Segre variety)  to $\PG(5,q^2)$ contains the transversal plane $\Pit$. The standard way used in the literature to extend  (see \cite{UnitalBook}) is using the variety $\K'=V(g_{0,0},g_{0,1},g_{1,0},g_{1,1},g_{2,0},g_{2,1})$.

 In Section~\ref{sec:lines-plane} we consider the Bose representation of an  $\Fq$-subplane and its relation to the Segre variety and its extension. In Section~\ref{sec:conic}, we consider an  $\Fq$-conic of  an  $\Fq$-subplane, and its representation on the Segre variety and the extended Segre variety.
 In light of this discussion, we are careful in this  work to describe the variety of $\PG(8,q)$ that we extend.

%
%
%
%
%

\subsection{Application to conics in the Bose representation}\Label{conic-sec}

We use Theorem~\ref{lem:cone} to show that a non-degenerate conic of $\PG(2,q^3)$ corresponds to the intersection of three quadrics in $PG(8,q)$. Further, we show that the points of the variety-extension lie on          a set of planes. 

\begin{theorem}\Label{Fqqq-conic-Bose} Let $\bar{ \O}$ be a non-degenerate conic  in $\PG(2,q^3)$. Consider the Bose representation.
\begin{enumerate}
\item In  $\PG(8,q)$, the pointset of  $\Bo{{\O}}$ coincides with the pointset of a variety $\VO=\Q_0\cap\Q_1\cap\Q_2$, where $\Q_1,\Q_2,\Q_3$ are quadrics of $\PG(8,q)$.
\item In $\PG(8,q^3)$, the variety-extension  $\VO\star=\Qzerostar\cap\Qonestar\cap\Qtwostar$ has pointset which coincides with the points  on the planes 
$\{\langle X,Y^{q},Z^{q^2}\rangle\st X, Y,Z\in\O\}.$
\item In $\PG(8,q^6)$, the variety-extension  $\VO\blackstar=\Qzerostarstar\cap\Qonestarstar\cap\Qtwostarstar$ has pointset which coincides with the points  on the planes 
$\{\langle X,Y^{q},Z^{q^2}\rangle\st X, Y,Z\in\O\blackstar\}$, where $\O\blackstar$ is the quadratic extension of the conic $\O$ to $\Pit\blackstar\subset\PG(8,q^6)$. 
\end{enumerate}
\end{theorem}

\begin{proof}
 Let $\bar{ \O}$ be a non-degenerate conic  in $\PG(2,q^3)$, so $\bar \O$ is the set of points satisfying a homogeneous equation $F(x,y,z)=0$ of degree 2 over $\Fqqq$. As in Theorem~\ref{lem:cone}, we can write $F(x,y,z)=f(x_0,x_1,x_2,y_0,y_1,y_2,z_0,z_1,z_2)
=f_0+ f_1\tau+ f_2\tau^2$ where $f_0,f_1,f_2$ are homogeneous equations of degree 2 over $\Fq$.  
Hence the set of points of $\PG(8,q)$ satisfying $f_i=0$ form a quadric denoted $\Q_i$, $i=0,1,2$. 
Further, 
by Theorem~\ref{lem:cone}(1),
in $\PG(8,q)$, the pointset of $\Bo{\O}$ form a variety $\V(\O)=\Q_0\cap\Q_1\cap\Q_2$.
The variety-extension to $\PG(8,q^3)$ is $\V(\O)\star=\Qzerostar\cap\Qonestar \cap\Qtwostar$, where $\Q_i\star$ is now a quadric of $\PG(8,q^3)$.  Moreover, $\V(\O)\star$ is   also defined by any three independent quadrics in the set $\lambda_0\Q_0+\lambda_1\Q_1+\lambda_2\Q_2$ with $\lambda_0,\lambda_1,\lambda_2\in\Fqqq$ (where $\lambda_0\Q_0+\lambda_1\Q_1+\lambda_2\Q_2$ denotes the quadric with homogeneous equation $\lambda_0f_0+\lambda_1f_1+\lambda_2f_2$).
Consider the following three independent  quadrics of $\PG(8,q^3)$, $\T_0=
\Qzerostar +\tau\Qonestar +\tau^2\Qtwostar $, 
$\T_1=
\Qzerostar +\tau^q\Qonestar +\tau^{2q}\Qtwostar $,
$\T_2=
\Qzerostar +\tau^{q^2}\Qonestar +\tau^{2q^2}\Qtwostar $.
If we denote the equation of $\T_0$ as $g=f_0+\tau f_1+\tau^2 f_2$, then $\T_1$ has equation $g^q$ and $\T_2$ has equation $g^{q^2}$. 
By 
Theorem~\ref{lem:cone}, $\T_0$ is a cone with base $\O$ in $\
\Pit$ and vertex $\langle \Pit^q,\Pit^{q^2} \rangle$. 
 Similarly $\T_1$ is a cone with base $\O^q$ and vertex $\langle \Pit,\Pit^{q^2} \rangle$,  and $\T_2$ is a cone with base  $\O^{q^2}$ and vertex $\langle \Pit,\Pit^{q} \rangle$.  
The intersection of these three cones is the set of T-planes that contain a point of $\O$, a point of $\O^q$ and a point of $\O^{q^2}$, as in part 2.

For part 3, the extension of $\V(\O)$ to $\PG(8,q^6)$ is $\V(\O)\blackstar=\T_0\blackstar\cap\T_1\blackstar\cap\T_2\blackstar$. 
In $\PG(8,q^3)$,  the quadric $\T_0$ is a cone with base $\O$ in $\
\Pit$ and vertex $\langle \Pit^q,\Pit^{q^2} \rangle$. Hence the   quadric extension $\T_0\blackstar$ is a cone with base $\O\blackstar$ in $\
\Pit\blackstar$ and vertex $\langle \Pit^q,\Pit^{q^2} \rangle\blackstar=\langle (\Pit^q)\blackstar,(\Pit^{q^2} )\blackstar\rangle$. 


We can similarly describe the quadrics $\T_1\blackstar,\T_2\blackstar$ as cones. The intersection of these three cones is the set of T-planes that contain a point of $\O\blackstar$, a point of $(\O^q)\blackstar=(\O\blackstar)^q$ and a point of $(\O^{q^2})\blackstar=(\O\blackstar)^{q^2}$. That is, the set of planes of form $\langle X,Y^q,Z^{q^2}\rangle$ where $ X,Y,Z\in \O\blackstar$, proving part 3.
\end{proof}

   \section{Sublines and
   subplanes  in the Bose representation}\Label{sec:lines-plane}

In this section we determine  the Bose representation of  $\Fq$-sublines and  $\Fq$-subplanes   of $\PG(2,q^3)$. We show in Section~\ref{sec:subline} that an $\Fq$-subline of $\PG(2,q^3)$ corresponds to a 2-regulus contained in a 5-space; and that an $\Fq$-subplane corresponds to a Segre variety $\S_{2;2}$. Moreover, we later wish to characterise $\Fq$-conics, and in order to do so, we  describe   in Section~\ref{sec6.2}  the extensions of a 2-regulus and a $\S_{2;2}$ to $\PG(8,q^3)$ and $\PG(8,q^6)$. 
     
     \subsection{$\Fq$-sublines  and subplanes in $\PG(8,q)$}\Label{sec:subline}

  Let $\bar \pi$ be an $\Fq$-subplane of  $\PG(2,q^3)$, we show that the  Bose representation is a Segre variety.

\begin{theorem}\Label{pi0-Bose} Let $\bar \pi$ be a $\Fq$-subplane of $\PG(2,q^3) $ then  in $\PG(8,q)$, the  planes $\{\Bo{X}\st X\in\pi\}$ 
of the Bose representation $\Bo{\pi}$  form one system of maximal spaces of a Segre variety $\S_{2;2}$. 
  \end{theorem}

\begin{proof} The $\Fq$-subplane $\bar\pi$ corresponds to an $\Fq$-subplane $\pi$ of  the transversal plane $\Pit$ in $\PG(8,q^3)$. Let $P_1,P_2,P_3,P_4$ be a quadrangle of $\pi$, then  the corresponding Bose planes $\Bo{P_i}=\langle P_i,P_i^q,P_i^{q^2}\rangle\cap\PG(8,q)$, $i=1,\ldots,4$ are four planes of $\PG(8,q)$, any three of which generate $\PG(8,q)$. 
By Lemma~\ref{lem:HT2} there is a unique Segre variety $\V=\S_{2;2}$. Denote the two systems of maximal spaces by  $\R,\R'$, with $\Bo{P_1},\ldots,\Bo{P_4}\in\R'$. We will  show that $\R'=\{\Bo{X}\st X\in\pi\}$.  

Consider the extension of the Segre variety $\V$ to a Segre variety $\Vstar$ of $\PG(8,q^3)$. Denote the two systems of maximal spaces (planes) of $\V\star$ by $\R\star$ and $(\R')\star$. So $\Bo{P_i}\star\in(\R')\star$, $i=1,\ldots,4$, and for $\pi_j\in\R$, we have $\pi_j\star\in\R\star$. As $\Pit$ is a plane of $\PG(8,q^3)$ that meets $\Bo{P_i}\star\in(\R')\star$, $i=1,\ldots,4$, we have $\Gamma\in\R\star$. Let $\alpha$ be a plane of $\R'$, then in $\PG(8,q^3)$, $\alpha\star$ is a plane of $(\R')\star$, and so $\alpha$ meets $\Pit$ in a point $A$. As $\alpha$ is a plane of $\PG(8,q)$, the points $A^q$ and $A^{q^2}$ lie in $\alpha$, and so $\alpha=\langle A,A^q,A^{q^2}\rangle\cap\PG(8,q)=\Bo{A}$.

As $\V$ is a Segre variety, the planes in $\R'$ are a scroll ruled by a homography, hence the set of planes $\{\alpha\star\st \alpha\in\R'\}$ of $\PG(8,q^3)$ are ruled by a homography. 
Thus  the set of points $\{\alpha\star\cap \Pit\st \alpha\in\R'\}$ form an $\Fq$-subplane $\pi'$ of $\Pit$. As both $\pi$ and $\pi'$ contain the quadrangle $P_1,\ldots,P_4$, we have $\pi'=\pi$.   
  Hence $\R'$ is the set of Bose planes $\Bo{X}$ for $X\in\pi$. 
 \end{proof}

  We use this to determine the Bose representation of an $\Fq$-subline. 
      Let $\bar b$ be an $\Fq$-subline of the line $\bar{\ell_b}$ in $\PG(2,q^3)$. In  
      the Bose representation in $\PG(8,q)$, we have $\Bo{b}=\{\langle Y,Y^q,Y^{q^2}\rangle \st Y\in b\}\cap\PG(8,q)$.  
      We show that this set of $q+1$ planes forms a 2-regulus. 

\begin{theorem}\Label{subline-simple} Let $\bar b$ be a $\Fq$-subline  lying on the line $\bar {\ell_b}$ of $\PG(2,q^3)$, then the planes of $\lbose b\rbose$ form a 2-regulus of
the   5-space $\Pi_{b}=\langle \ell_b,\ell_b^q,\ell_b^{q^2}\rangle\cap\PG(8,q)$.\end{theorem}

 \begin{proof} 
  Let $\bar\pi$ be any $\Fq$-subplane of $\PG(2,q^3)$ which contains $\bar b$ as a line. In the Bose representation, $\pi$ is an $\Fq$-subplane of $\Pit$ that contains the $\Fq$-subline $b$. As $b=\pi\cap\ell_b$,
 $\Bo{b}=\Bo{\pi}\cap\Bo{\ell_b}$. These $q+1$ planes all lie in the 5-space 
   $\Pi_{b}=\langle \ell_b,\ell_b^q,\ell_b^{q^2}\rangle\cap\PG(8,q)$, so $\Bo{b}=\Bo{\pi}\cap\Pi_b$. By Theorem~\ref{pi0-Bose}, the planes of $\Bo{\pi}$   form one system of  maximal planes of a Segre variety $\S_{2;2}$. By \cite[Thm 4.109]{HT-new}, the 5-space $\Pi_b$ meets the $\S_{2;2}$ in a Segre variety $\S_{1;2}$. That is, the planes of $\Bo{b}$ are the system of maximal planes of a Segre variety $\S_{1;2}$, and so are a 2-regulus. 
\end{proof}

       \subsection{Extending to $\PG(8,q^3)$}\Label{sec6.2}

We now
look at the extension of the varieties in Theorem~\ref{pi0-Bose} and Theorem~\ref{subline-simple} to 
 $\PG(8,q^3)$ and $\PG(8,q^6)$. We recall two collineations of $\PG(8,q^6)$, namely 
$X=(x_0,\ldots,x_8)\longmapsto X^q=(x_0^q,\ldots,x_8^q)$ and  $\mathsf e\colon X=(x_0,\ldots,x_8)\mapsto X^{\mathsf e}=(x_0^{q^3},\ldots,x_8^{q^3})$.

\begin{theorem}\Label{subplane-star-3} Let $\bar \pi$ be an $\Fq$-subplane   of $\PG(2,q^3)$. In the Bose representation, let ${\mathsf c_\pi}$ be the  collineation of order 3 acting on the points of $\Pit$ which fixes $\pi$ pointwise as defined in (\ref{defcpi}). Let $\VB$ denote the Segre variety  $\S_{2;2}$ whose pointset coincides with the pointset of $\Bo{\pi}$. 
\begin{enumerate}
\item In  $\PG(8,q^3)$,
$\VB\star$ is a Segre variety, with one system of maximal spaces the planes   $\{ \lround X\rround_\pi\st X\in\Pit\}$ where 
 $$ \lround X\rround_\pi=\langle\, X,\ (\conjBsq{X})^q, \ (\conjB{X})^{q^2}\,\rangle.$$

\item In  $\PG(8,q^6)$,
$\VB\blackstar$ is a Segre variety,  with one system of maximal spaces the planes    $\{\lround X\rround_\pi\st X\in\Pit\blackstar\}$ where 
 $$\lround X\rround_\pi=\big\langle\, X, (X^{{\mathsf c}^2_\pi\mathsf e})^q, (X^{\mathsf c_\pi\mathsf e})^{q^2}\,\big\rangle=
 \big\langle\, X, (X^{{\mathsf c}^5_\pi})^q, (X^{\mathsf c^4_\pi})^{q^2}\,\big\rangle.$$
\end{enumerate}
 \end{theorem}

\begin{proof}
We first show that the planes of $\Bo{\pi}$ form a scroll. 
Without loss of generality, let $\bar\pi=\PG(2,q)=\{(x,y,z)\st x,y,z\in\Fq, \ \textup{not all }0\}$. 
In the Bose representation, $\pi=\{xA_0+yA_1+zA_2\st x,y,z\in \Fq, \ \textup{not all }0\}$ is an $\Fq$-subplane of the transversal plane $\Pit$.
For a point
$X=xA_0+yA_1+zA_2\in\pi$, (so $x,y,z\in\Fq$) we have  $\Bo{X}=\langle \ X,\ X^q,\ X^{q^2}\ \rangle\cap\PG(8,q)$, that is, $$\Bo{X}\star=\langle \   xA_0+yA_1+zA_2,\ xA_0^q+yA_1^q+zA_2^q,\ 
 xA_0^{q^2}+yA_1^{q^2}+zA_2^{q^2} \ \rangle. 
$$
So the planes of $\Bo{\pi}$ form a scroll denoted $\scroll_\pi$, where the associated homographies are the identity.

By Theorem~\ref{pi0-Bose}, in $\PG(8,q)$, the planes of $\Bo{\pi}$ form one system of maximal spaces of a Segre variety $\S_{2;2}$, denoted  by $\V(\Bo{\pi})$.  The variety-extension of $\VB$ to $\VB\star$ in $\PG(8,q^3)$ is also a Segre   variety $\S_{2;2}$. Hence $\VB\star$ consists of two systems  of maximal subspaces, each ruled by a homography, that is, each system of planes of $\VB\star$ forms a scroll.
 As a homography is uniquely determined by a quadrangle, 
 a scroll has a unique scroll-extension. Hence the Segre variety $\VB\star$  has one system of maximal subspaces  being the same set of planes as the scroll-extension of $\scroll_\pi$. 
 
 We now determine the planes of the scroll-extension of $\scroll_\pi$  to 
 $\PG(8,q^3)$. As the homographies of $\scroll_\pi$ are the identity, the scroll-extension  is the set of planes  $$\{ \langle \   xA_0+yA_1+zA_2,\ xA_0^q+yA_1^q+zA_2^q,\ 
 xA_0^{q^2}+yA_1^{q^2}+zA_2^{q^2} \,\rangle\st x,y,z\in\Fqqq, \textup{ not all }0\}.
$$
Using the calculations from (\ref{eqn-1}) in Section~\ref{sec-conj}, this is the set of planes $$ \{\langle\, X,\ (\conjBsq{X})^q,\ (\conjB{X})^{q^2}\,\rangle\st X\in\Gamma\}.$$
Hence the Segre variety $\VB\star$ has as one system of maximal spaces the planes $\{\lround X\rround_\pi\st X\in \Gamma\}$, with $\lround X\rround_\pi=\langle\, X,\ (\conjBsq{X})^q,\ (\conjB{X})^{q^2}\,\rangle$, proving part 1. 

The proof of part 2 is similar. The scroll-extension  of $\scroll_\pi$ 
 to $\PG(8,q^6)$ is the set of planes $$ \{\langle \   xA_0+yA_1+zA_2,\ xA_0^q+yA_1^q+zA_2^q,\ 
 xA_0^{q^2}+yA_1^{q^2}+zA_2^{q^2} \,\rangle
\st x,y,z\in\Fqqqqqq\textup{ not all }0\}.$$
Using the calculations from  (\ref{eqn-2}) in Section~\ref{sec-conj}, this is the set of planes $$
\{ \langle\, X, (X^{{\mathsf c}^2_\pi\mathsf e})^q, (X^{\mathsf c_\pi\mathsf e})^{q^2}\,\rangle\st X\in\Pit\blackstar\}.$$
Hence the   Segre variety  $\VB\blackstar$ has as one system of maximal spaces the planes  $\{\lround X\rround_\pi\st X\in\Pit\blackstar\}$ with $\lround X\rround_\pi= \langle\, X, (X^{{\mathsf c}^2_\pi\mathsf e})^q, (X^{\mathsf c_\pi\mathsf e})^{q^2}\,\rangle$, proving part 2. 
\end{proof}

 Note that for a point $X\in \pi$,  $\conjB{X}=X$ and  $\lround X\rround_\pi=\Bo{X}\star$, which  meets $\PG(8,q)$ in a plane $\Bo{X}$.  However, if $X\in\Pit\setminus \pi$, then by Lemma~\ref{lem:hypcong}, the plane $\lround X\rround_\pi$ does not meet any element of  $\{\lbose X\rbose\st X\in\Pit\}$. As the planes in the set $\{\lbose X\rbose\st X\in\Pit\}$ partition the points of $\PG(8,q)$, it follows that $\lround X\rround_\pi$ is disjoint from $\PG(8,q)$ for  $X\in\Pit\setminus \pi$.

We use Theorem~\ref{subplane-star-3} to look at an $\Fq$-subline $\bar b$, and determine the planes of the  extension of the 2-regulus $\Bo{b}$ to $\PG(8,q^3)$ and $\PG(8,q^6)$.

\begin{corollary}\Label{subline-star-3} Let $\bar b$ be an $\Fq$-subline  lying on the line $\bar {\ell_b}$ of $\PG(2,q^3)$. In the Bose representation, let ${\mathsf c_b}$ be the  collineation of order 3 acting on the points of $\ell_b$ which fixes $b$ pointwise as defined in (\ref{defcb}). The 2-regulus  $\Bo{b}$ can be extended to a unique 2-regulus of 
 $\PG(8,q^3)$ with  planes    $\{ \lround X\rround_b\st X\in\ell_b\}$; 
 and to a unique 2-regulus   of $\PG(8,q^6)$ with planes $\{\lround X\rround_b\st X\in\ell^{\blackstar}_b\}$ where 
$$\lround X\rround_b=
 \big\langle\, X, (X^{{\mathsf c}^5_b})^q, (X^{\mathsf c^4_b})^{q^2}\,\big\rangle.$$
 \end{corollary}
 
 \begin{proof}
 Note that if $X\in\ell_b$, then $X^{{\mathsf c}^5_b}=X^{{\mathsf c}^2_b}$ and $X^{\mathsf c^4_b}=X^{\mathsf c_b}$. Further, if $X\in b$, then $X^{{\mathsf c}^5_b}=X^{\mathsf c^4_b}=X$.
 If $\pi$ is an $\Fq$-subplane of $\Pit$, then $\mathsf c_b$ coincides with $\mathsf c_\pi$ restricted to $\ell_b$ if and only if $b$ is a line of $\pi$. Hence letting $\pi$ be an  $\Fq$-subplane of $\Pit$ such that $b$ is a line of $\pi$, we can intersect the results of 
 Theorem~\ref{subplane-star-3} with the 5-space $\Pi_b=\langle \ell_b,\ell_b^q,\ell_b^{q^2}\rangle$ to obtain the required result. 
 \end{proof}

   \section{$\Fq$-conics in the Bose representation}\Label{sec:conic}
   
We define an $\Fq$-conic of $\PG(2,q^2)$ to be a non-degenerate conic in an $\Fq$-subplane of $\PG(2,q^2)$. That is, an $\Fq$-conic is projectively equivalent to a set of points in $\PG(2,q)$ that satisfy a non-degenerate homogeneous quadratic equation over $\Fq$.  
We  determine the Bose representation of an $\Fq$-conic $\bar\C$ of $\PG(2,q^3)$.
An $\Fq$-conic  $\bar\C$ of $\PG(2,q^3)$ corresponds to an $\Fq$-conic in the transversal plane $\Pit$ denoted $\C$, and $\Cplus$ denotes the unique $\Fqqq$-conic of $\Pit$ containing $\C$. The quadratic extension of the non-degenerate conic $\Cplus\subset\Pit$ to the extended transversal plane $\Pit\blackstar\cong\PG(2,q^6)$
is  a non-degenerate conic which we denote by $\C^{{\rm \plus\!\plus}}$.

\begin{theorem}\Label{conic-subplane-B}
    Let $\bar \C$ be an $\Fq$-conic in the $\Fq$-subplane $\bar\pi$ of $\PG(2,q^3)$, and consider the Bose representation of $\PG(2,q^3)$ in $\PG(8,q)$. 
\begin{enumerate}
\item 
In $\PG(8,q)$, the planes of  $\Bo{\C}$ form a scroll of $\PG(8,q)$, and the pointset of $\Bo{\C}$ forms a variety $\VC=\V^6_3$. 

\item 
In $\PG(8,q^3)$, the points of the variety  $\VC\star $ coincide with the points on the planes $  \{\lround X\rround_\pi\st X\in\Cplus\},
$  which form a scroll.
\item In $\PG(8,q^6)$,  the  points of the variety  $\VC\blackstar$  coincide with the points on the planes $ \{\lround X\rround_\pi\st X\in \C^{{\rm \plus\!\plus}}  \}$, which form a scroll.

\end{enumerate}
\end{theorem}

\begin{proof}
Let $\C$ be an $\Fq$-conic in an $\Fq$-subplane $\pi$ of $\Pit$, so $\C=\pi\cap\Cplus$. 
By definition, in $\PG(8,q)$, $\Bo{\C}=\{\langle X,X^q,X^{q^2}\rangle \cap\PG(8,q)\st X\in\C\}$, $\Bo{\pi}=\{\langle X,X^q,X^{q^2}\rangle\cap\PG(8,q)\st X\in\pi\}$ and $\Bo{\Cplus}=\{\langle X,X^q,X^{q^2}\rangle\cap\PG(8,q)\st X\in\Cplus\}$, so 
\begin{eqnarray}\label{eqn:c1}
\Bo{\C}=\Bo{\Cplus}\cap \Bo{\pi}.
\end{eqnarray}
By Theorem~\ref{pi0-Bose}, the planes of  $\Bo{\pi}$ form one system of maximal subspaces of a Segre variety $\S_{2;2}$, and so form a scroll. As the planes of  $\Bo{\C}$ are a subset of the planes of $\Bo{\pi}$, the  planes of $\Bo{\C}$ form  a scroll, ruled by the same homography as for the scroll $\Bo{\pi}$. 
Hence by Lemma~\ref{conic-v63}, the pointset of $\Bo{\C}$ forms a variety $\V^6_3$.

To prove part 2, we first need to describe this variety $\V^6_3$ in more detail. 
By Theorem~\ref{Fqqq-conic-Bose}, the pointset of $\Bo{\Cplus}$ forms a variety $\VCplus=\Q_1\cap\Q_2\cap\Q_3$ where $\Q_i$ is a quadric   with homogenous equation $f_i=0$ of degree two over $\Fq$, $i=1,2,3$.
  By \cite{HT-new}, a Segre variety $\S_{2;2}$ is the intersection of six quadrics, so by Theorem~\ref{pi0-Bose}, the pointset of $\Bo{\pi}$ forms a variety $\VB=\Q_4\cap\cdots\cap\Q_9$ where $\Q_i$ is a quadric with homogenous equation $f_i=0$ of degree two over $\Fq$, $i=4,\ldots,9$. So by (\ref{eqn:c1}), the pointset of $\Bo{\C}$ coincides with the pointset of  a variety $\VC$ which is the intersection of nine quadrics, namely $
  \VC=(\Q_1\cap\Q_2\cap\Q_3)\cap(\Q_4\cap\cdots\cap\Q_9).
$

  The variety extension of $\VC$ is $\VC\star=\Qonestar\cap\cdots\cap\Qninestar$, so in particular, 
  \begin{eqnarray}\label{eqn:yay}
  \VC\star=\VCplus\star\cap\VB\star.
  \end{eqnarray}

We now determine the points of $\VC\star$. 
By Theorem~\ref{Fqqq-conic-Bose}, the points of  $\VCplus\star$ are the points of $\PG(8,q^3)$ on the planes 
\begin{eqnarray}\label{eqn:cplus}
\{\langle X,Y^q,Z^{q^2}\rangle\st X,Y,Z\in\Cplus\}.
\end{eqnarray}
  By Theorem~\ref{subplane-star-3}, the points of  $\VB\star$ are the points  of $\PG(8,q^3)$  on the planes \begin{eqnarray}\label{eqn:B}
  \{\lround X\rround_\pi \st X\in\pi\}.
  \end{eqnarray}
The planes in (\ref{eqn:cplus}) and   (\ref{eqn:B}) are T-planes, so  by  Lemma~\ref{lem:hypcong}, two planes in (\ref{eqn:cplus}) and   (\ref{eqn:B}) either coincide, are disjoint, or meet in a T-point or a T-line.  Thus by (\ref{eqn:yay}),  $\VC\star$ consists of points on the set of planes which are in both (\ref{eqn:cplus}) and   (\ref{eqn:B}). That is,  $\VC\star$ consists of the points  of $\PG(8,q^3)$  on  the planes $\{\lround X\rround_\pi\st X\in\Cplus\}$.

 By Theorem~\ref{subplane-star-3}, the points of the variety-extension $\VB\star$ 
 are one system of maximal subspaces of a Segre variety and so form a scroll.
 As the planes of  $\VC\star$ are a subset of the planes of $\VB\star$, the  planes of $\VC\star$ form  a scroll, ruled by the same homography as for the scroll $\VB\star$. This completes the proof of part 2. Part 3 is similar.
\end{proof}


\bigskip\bigskip

{\bfseries Author information}

S.G. Barwick. School of Mathematical Sciences, University of Adelaide, Adelaide, 5005, Australia.
susan.barwick@adelaide.edu.au

W.-A. Jackson. School of Mathematical Sciences, University of Adelaide, Adelaide, 5005, Australia.
wen.jackson@adelaide.edu.au

P. Wild. Royal Holloway, University of London, TW20 0EX, UK. peterrwild@gmail.com


\begin{thebibliography}{999}


%


\bibitem{UnitalBook}
S.G.~Barwick and G.L.~Ebert.
 {\em Unitals in Projective Planes}.
 Springer Monographs in Mathematics. Springer, New York, 2008.


\bibitem{BJ-FFA}
S.G. Barwick and W.-A. Jackson. Sublines and subplanes of PG$(2,q^3)$ in the Bruck-Bose representation in PG$(6,q)$.
{\em Finite Fields Appl.,} 18 (2012) 93--107.


%
%
%
%
%
%


 \bibitem{BJW-4}  S.G. Barwick, W.-A.~Jackson and P. Wild. A characterisation of $\Fq$-conics of $\PG(2,q^3)$.  Preprint.


\bibitem{Bose} R.C. Bose. 
On a representation of the {B}aer subplanes of the {D}esarguesian
  plane {$\PG(2,q^{2})$} in a projective five dimensional space. \emph{Teorie
  Combinatorie}, vol.~I, Accad. Naz. dei Lincei, Rome, 1976, (Rome, 1973),
  pp.~381--391.


%
%
%
%
%
%
%
%
%
%
%
\bibitem{CasseOKeefe} L.R.A. Casse and C.M. O'Keefe. Indicator sets for $t$-spreads of $\PG((s+1)(t+1)-1,q)$. {\em Bollettino U.M.I.}, {\bf 4} (1990) 13--33.

%


%
  \bibitem{HT-new} J.W.P.\ Hirschfeld and J.A.\ Thas. {\em General Galois Geometries.} 2nd edition. Springer-Verlag, London, 2016.
%
%
%
%



\bibitem{SR}
J.~G. Semple and L.~Roth. \emph{Introduction to Algebraic Geometry}. Oxford
  University Press,  1949.


\end{thebibliography}
\end{document}